\documentclass[a4paper,12pt]{article}

\usepackage[T1]{fontenc}
\usepackage{amsmath,amssymb,amsthm,mathtools}
\usepackage{graphicx}
\usepackage{hyperref}
\usepackage{enumitem}

\allowdisplaybreaks

\usepackage{color}
\usepackage{xcolor}

\newtheorem{thm}{Theorem}[section]
\newtheorem{proposition}[thm]{Proposition}
\newtheorem{lemma}[thm]{Lemma}
\newtheorem{theorem}[thm]{Theorem}

\DeclarePairedDelimiter\norm\lVert\rVert
\newcommand\nularg{{}\cdot{}}
\DeclareMathOperator\per{per}
\newcommand\dd{\mathrm{d}}
\newcommand{\Ast}{\scalebox{2}{$\ast$}}

\usepackage{authblk}

\renewenvironment{abstract}{%
	\vspace{6pt}%
	\begin{center}%
		\begin{minipage}{320pt}%
			\small%
			\begin{center}%
				\textbf{Abstract}%
			\end{center}%
		}{\end{minipage}\end{center}}
\newcommand{\keywords}[1]{%
	\begin{center}%
		\begin{minipage}{320pt}%
			\textbf{Keywords:}~{#1}
		\end{minipage}%
	\end{center}%
}
\newcommand{\msc}[2]{%
	\begin{center}%
		\begin{minipage}{320pt}%
			\medskip\small%
			2010~Mathematics~Subject~Classification:~Primary \uppercase{#1} \\%
			\phantom{2010~Mathematics~Subject~Classification:~}Secundary \uppercase{#2}%
		\end{minipage}%
	\end{center}%
}

\begin{document}
	\title{Pair correlation functions and limiting distributions of	iterated cluster point processes}
	\author{Jesper M{\o}ller and Andreas D. Christoffersen}
	\affil{Department of Mathematical Sciences, Aalborg University}
	\date{}
	\maketitle
	
\begin{abstract}
 	We consider a Markov chain of point processes such that each state is a superposition of an independent cluster process with the previous state as its centre process together with some independent noise process and a thinned version of the previous state. 
	The model extends earlier work by Felsenstein and Shimatani describing a reproducing population. We discuss when closed term expressions of the first and second order moments are available for a given state.	In a special case it is known that the pair correlation function for these type of point processes converges as the Markov chain progresses, but it has not been shown whether the Markov chain has an equilibrium distribution with this, particular, pair correlation function and how it may be constructed. Assuming the same reproducing system, we construct an equilibrium distribution by a coupling argument.
\end{abstract}


\msc{60G55; 60J20}{60D05; 62M30} %; 68U20}

%with point processes in $\mathbb R^d$, the $d$-dimensional Euclidean space, viewing a point process as a random set $X\subset\mathbb R^d$ which is almost surely locally finite, %with no accumulation points (hence $X$ is a random closed set; 
%that is, $X$ has almost surely a finite number of points within any bounded subset of $\mathbb R^d$ (for measure theoretical details, see e.g.\ \cite{daley:vere-jones:03}
%or \cite{MW2004}). Recall that $X$ is stationary if its distribution is invariant under translations in $\mathbb R^d$, and then its intensity $\rho_X\in[0,\infty]$ is the mean number of points in $X$ falling in any Borel subset of $\mathbb R^d$ of unit volume (if $\rho_X=0$, then $X$ is almost surely empty; if $\rho_X<\infty$, then $X$ is almost surely finite). 
%In particular, $X$ is a stationary Poisson process with finite intensity $\rho_X$ if for any Borel set $B\subset \mathbb R^d$ of volume $|B|\in(0,\infty)$, the number of points in $X\cap B$ is Poisson distributed with mean $\rho_X|B|$ (a degenerated Poisson distribution if $\rho_X=0$), and conditional on $\#(X\cap B)=n>0$, the $n$ points are independent uniformly distributed in $B$. 

\section{Introduction}
This paper deals with a discrete time Markov chain of point processes
$G_0,G_1,\ldots$ in the $d$-dimensional Euclidean space $\mathbb R^d$,
where the chain describes a reproducing population and we refer to
$G_n$ as the $n$th generation (of points).  We make the following
assumptions.  Any point process considered in this paper will be
viewed as a random subsets of $\mathbb R^d$ which is almost surely
locally finite, that is, the point process has almost surely a finite
number of points within any bounded subset of $\mathbb R^d$ (for
measure theoretical details, see e.g.\ \cite{daley:vere-jones:03} or
\cite{MW2004}). Recall that a point process $X\subset\mathbb R^d$ is
stationary if its distribution is invariant under translations in
$\mathbb R^d$, and then its intensity $\rho_X\in[0,\infty]$ is the
mean number of points in $X$ falling in any Borel subset of
$\mathbb R^d$ of unit volume.
% (if $\rho_X=0$, then $X$ is almost surely empty; if $\rho_X<\infty$,
% then $X$ is almost surely finite).
Now, for generation 0, $G_0$ is stationary with intensity
$\rho_{G_0}\in(0,\infty)$. Further, for generation $n=1,2,\dots,$
conditional on the previous generations $G_0,\ldots,G_{n-1}$, we
obtain $G_n$ by four basic operations for point processes:
\begin{enumerate}
	\item[(a)] Independent clustering: To each point $x\in G_{n-1}$ is
	associated a (non-centred) cluster $Y_{n,x}\subset\mathbb
	R^d$. These clusters are independent identically distributed (IID)
	finite point processes and they are independent of
	$G_0,\ldots,G_{n-1}$. The cardinality of $Y_{n,x}$ has finite mean
	$\beta_n$ and finite variance $\nu_n$ and is independent of the
	points in $Y_{n,x}$ which are IID, with each point following a
	probability density function (PDF) $f_n$.  We refer to $x+Y_{n,x}$
	(the translation of $Y_{n,x}$ by $x$) as the offspring/children
	process generated by the ancestor/parent $x$, and we let
	\begin{equation}\label{e:DefYn}
		Y_n=\bigcup_{x\in G_{n-1}}(x+Y_{n,x})
	\end{equation}
	be the independent cluster process given by the superposition of all
	offspring processes generated by the points in the previous
	generation $G_{n-1}$.
	\item[(b)] Independent thinning: For all $y\in\mathbb R^d$, let
	$B_{n,y}$ be IID Bernoulli variables which are independent of $Y_n$,
	$G_0,\ldots,G_{n-1}$, and all previously generated Bernoulli
	variables. Let $p_n=\mathrm P(B_{n,y}=1)$. For all $x\in G_{n-1}$,
	let
	\begin{equation*}
		W_{n,x}=\{y\in x+Y_{n,x}:B_{n,y}=1\}	
	\end{equation*}
	be the independent $p_n$-thinned point process of $x+Y_{n,x}$, and
	let
	\begin{equation}\label{e:DefWn}
		W_n=\bigcup_{x\in G_{n-1}}W_{n,x}
	\end{equation}
	be the independent $p_n$-thinned point process of $Y_n$. Note that with probability one, $W_n\cap G_{n-1}=\emptyset$, since by assumption on the cluster points the origin is not contained in $Y_{n,x}$.  
	\item[(c)] Independent retention: For all $x\in \mathbb{R}^d$, let $Q_{n,x}$ be IID Bernoulli variables which are independent of $Y_n$,
	$G_0,\ldots,G_{n-1}$, and all previously generated Bernoulli
	variables. Let $q_n=P(Q_{n,x} = 1)$ and  %For all $x\in G_{n-1}$, let 
	let
		\begin{equation*}
			G_{n-1}^{\rm thin} = \left\{x\in G_{n-1} : Q_{n,x} = 1\right\}
		\end{equation*}
		be the independent $q_n$-thinned point process of $G_{n-1}$. 
	\item[(d)] Independent noise: Let $Z_n\subset\mathbb R^d$ be a
	stationary point process with finite intensity $\rho_{Z_n}$ and 
	independent of $W_n$, $G_0,\ldots,G_{n-1}$, and $G_{n-1}^{\rm thin}$. Finally, let
	\begin{equation}\label{e:DefXn}
	G_n=W_n\cup G_{n-1}^{\rm thin} \cup Z_n
	\end{equation}
	where we interpret $Z_n$ as noise. For ease of presentation we assume with probability one that $W_n\cup G_{n-1}^{\rm thin}$ and $Z_n$ are disjoint. Thus $ W_n$, $ G_{n-1}^{\rm thin}$, and $Z_n$ are pairwise disjoint almost surely.
\end{enumerate}
When we later interpret our results, for any point $x\in G_{n-1}^{\rm thin}$, since $x\in G_{n-1}\cap G_n$, we consider $x$ 
  both as its own ancestor and its own child. 

Our model is an extension of the model in Shimatani's paper \cite{shimatani:10}, which
in turn is an extension of Mal{\'e}cot's model studied in
\cite{felsenstein:75} (we return to this in
Section~\ref{s:assumptionsliterature}, item (vii) and (viii)).  In
particular, our extension allows us to model cluster centres
exhibiting clustering or regularity, points from previous generations can be retaining, and the noise processes
can also exhibit clustering or regularity (i.e., they are not assumed to be Poisson processes). 
 For statistical applications, we have in
mind that $G_n$ may be observable (at least for some values of
$n\ge 1$) whilst $G_0$ and the cluster, thinning,  and superpositioning
procedures in items (a)--(b) and (d) are unobservable.  Our model may be of
relevance for applications in population genetics and community
ecology (see \cite{shimatani:10} and the references therein), for
analyzing tropical rain forest point pattern data with multiple scales
of clustering (see \cite{wiegandetal:07}), and for modelling proteins
with multiple noisy appearances in PhotoActivated Localization
Microscopy (PALM) (see \cite{Andersen:etal:17}). However, we leave it
for other work to study the statistical applications of our model and
results.

The paper is organized as follows. A discussion of the assumptions in
items (a)--(d) and the related literature is given in
Section~\ref{s:assumptionsliterature}. Section~\ref{s:moments} focuses
on the first and second order moment properties of $G_n$, that is, its
intensity and pair correlation function (PCF); we extend model cases and results in Shimatani's paper
\cite{shimatani:10} and show that tractable model cases for the PCF of
$G_0$ %(extending cases considered in \cite{shimatani:10}) 
are
meaningful in terms of Poisson and other point processes, including weighted permanental and weighted
determinantal point processes (which was not observed in
\cite{shimatani:10}). Section~\ref{s:samesystem} discusses limiting
cases of the PCF of $G_n$ as $n\rightarrow\infty$ when we have the
same reproduction system and under weaker conditions than in
\cite{shimatani:10}.  In particular, when natural conditions are
satisfied, we establish ergodicity of the Markov chain by using a
coupling construction and by giving a constructive description of the
chain's unique invariant distribution when extending the Markov chain
backwards in time.
% Furthermore, Section~\ref{s:Examples}....
Finally, Appendix~\ref{a:1} provides background knowledge on weighted
permanental and determinantal point processes, 
Appendix~\ref{a:2}
verifies some technical details, 
and Appendix~\ref{a:3} specifies an
algorithm for approximate simulation of the Markov chain's invariant
distribution.

\section{Assumptions and related work}\label{s:assumptionsliterature}

Items (i)--(iv) below comment on the model assumptions in items
(a)--(d).
\begin{enumerate}
	\item[(i)] The process $Y_n$ is a stationary independent cluster
	process \cite{daley:vere-jones:03} and we have the following
	special cases: If $G_{n-1}$ is a stationary Poisson process, $Y_n$
	is a Neyman-Scott process \cite{neyman:scott:58}; if in addition
	$\#Y_{n,x}$ follows a Poisson distribution, then $\beta_n=\nu_n$ and
	$Y_n$ is a shot-noise Cox process (SNCP; \cite{moeller:02}) driven
	by
	% that is, conditional on $G_{n-1}$ each cluster $Y_{n,x}$ is a
	% Poisson process with intensity function $\beta_nf_n(\cdot-x)$ and
	% $Y_n$ is a Poisson process with intensity function
	\begin{equation}\label{e:SNCP}
		\Lambda_n(x)=\beta_n\sum_{y\in G_{n-1}}f_n(x-y),\qquad x\in\mathbb R^d. 
	\end{equation} 
	This is a (modified) Thomas process \cite{thomas:49} if $f_n$ is
	the density of $d$ IID zero-mean normally distributed variates with
	variance $\sigma_n^2$ -- we denote this distribution by
	$N_d(\sigma_n^2)$ -- and it is a Mat{\'e}rn cluster process
	\cite{matern:60,matern:86} if instead $f_n$ is a uniform density
	of a $d$-dimensional ball with centre at the origin.  However, in
	many applications a Poisson centre process is not appropriate.  For
	instance, Van Lieshout \& Baddeley
	(2002)\nocite{lieshout:baddeley:01} considered a repulsive Markov
	point process model for the centre process, whereby it is easier to
	identify the clusters than under a Poisson centre process.
	% Then $Y_n$ is a special case of a generalised shot-noise Cox
	% process (\cite{moeller:torrisi:2004}).
	
	\item[(ii)] When $\beta_n\le\nu_n$, we may consider $Y_n$ as a stationary generalised shot-noise Cox process
	(GSNCP; see \cite{moeller:torrisi:2004}). In this model
	\eqref{e:SNCP} is extended to the case where $G_{n-1}$ is a general
	stationary point process and $Y_n$ is a Cox process driven
	by %conditional on $G_{n-1}$ is a Poisson process with intensity function
	\begin{equation*}\label{e:GSNCP}
		\Lambda_n(x)=\sum_{y\in G_{n-1}}\gamma_y
		k_n[\{(x-y)/b_y\}]/b_y^d,\qquad x\in\mathbb R^d, 
	\end{equation*}
	where $k_n$ is a PDF on $\mathbb R^d$, the $\gamma_y$ and the $b_y$
	for all $y\in G_{n-1}$ are independent positive random variables
	which are independent of $G_{n-1}$, and the $\gamma_y$ are
	identically distributed with mean $\beta_n$ and variance
	$\nu_n-\beta_n$ (as $\# Y_{n,x}$ has mean $\beta_n$ and variance
	$\nu_n=\mathrm E\{\mathrm{var}(\#
	Y_{n,x}|\gamma_y)\}+\mathrm{var}\{\mathrm E(\#
	Y_{n,x}|\gamma_y)\}=\beta_n+\mathrm{var}(\gamma_n)$). Further, $b_y$
	has an interpretation as a random band-width and
	\begin{equation*}
		f_n(x)=\mathrm E\left\{\frac{k_n(x/b_y)}{b_y^d}\right\}.
	\end{equation*}
	The general results for the
	intensity and PCF of $G_n$ in Section~\ref{s:moments} will be
	unchanged whether we consider this stationary GSNCP or the more
	general case in item (a)
	
	\item[(iii)] Clearly, there is no noise ($Z_n$ is empty with
	probability one) if $\rho_{Z_n}=0$. The case $\rho_{Z_n}>0$ may be
	relevant when not all points in a generation can be described as
	resulting from independent clustering and thinning as in (a)--(c).  Note that in
	item (d) we could without loss of generality assume
	$Z_1, Z_2, \ldots$ are independent.
	%, however, it will first be in Section~\ref{s:samesystem} that we assume they are IID.
	% , although the assumption of $Z_n$ being a stationary Poisson
	% process may be a rough but convenient simplification.
	Further, we introduce the thinning of $Y_n$ in item (b) only for
	modelling purposes and for comparison with \cite{shimatani:10}; from a mathematical point of view this thinning
	could be omitted if in item (a) we replace each cluster $Y_{n,x}$
	by what happens after the independent thinning: Namely that
	independent thinned clusters $Y_{n,x}^{\mathrm{th}}$ appear so that
	$\# Y_{n,x}^{\mathrm{th}}$ has mean
	$\beta_n^{\mathrm{th}}=\beta_np_n$ and variance
	$\nu_n^{\mathrm{th}}=\beta_np_n-\beta_n p_n^2+\nu_np_n^2$ and is
	independent of the points in $Y_{n,x}^{\mathrm{th}}$ which are IID
	with PDF $f_n$, whereby $W_n$ and
	$Y_n^{\mathrm{th}}:=\cup_{x\in G_{n-1}}(x+Y_{n,x}^{\mathrm{th}})$
	are identically distributed.
	
	\item[(iv)] Assuming for $n=1,2,\ldots$ no thinning of $Y_n$ ($p_n=1$), an equivalent description
	of items (a) and (c)--(d) is given in terms of the Voronoi tessellation
	generated by $G_{n-1}$: For $x\in G_{n-1}$, let $C(x|G_{n-1})$ be
	the Voronoi cell associated to $x$ and consisting of all points in
	$\mathbb R^d$ which are at least as close to $x$ than to any
	other point in $G_{n-1}$ (with respect to usual distance in
	$\mathbb R^d$). With probability one, since $G_{n-1}$ is stationary
	and non-empty, each Voronoi cell is bounded and hence its volume is
	finite (see e.g. \cite{moeller:89a,moeller:94}). Thus we can set
	\begin{equation*}
		G_n=\bigcup_{x\in G_{n-1}}(x+G_{n,x})\bigcup G_{n-1}^{\rm thin}
	\end{equation*}
	where conditional on $G_{n-1}$ and for all $x\in G_{n-1}$, the
	$G_{n,x}$ are independent of $G_{n-1}^{\rm thin}$ and they are IID finite point processes with a distribution as
	follows: $\#G_{n,x}$ has mean $\beta_n+|C(x|G_{n-1})|\rho_{Z_n}$,
	variance $\nu_n+|C(x|G_{n-1})|\rho_{Z_n}$, and is independent of the
	points in $G_{n,x}$, where $|\cdot|$ denotes volume. The points in $G_{n,x}$ are i.i.d., each following a
	mixture distribution so that with probability
	$\beta_n/\{\beta_n+|C(x|G_{n-1})|\rho_{Z_n}\}$ the PDF is $f_n$ and
	else it is a uniform distribution on $C(x|G_{n-1})$.
\end{enumerate}

In items (v)--(vi) below we discuss earlier work on the model for
$G_0,G_1,\ldots$, where $G_0$ is a stationary Poisson process, all
$G_n=Y_n$ for $n\ge1$ (i.e., no thinning, no retention, and no noise), $f_n=f$ and
$\beta_n=\beta$ do not depend on $n\ge1$. We may refer to this as a
replicated SNCP. Frequently in the literature, a so-called replicated
Thomas process is considered, that is, $f\sim N_d(\sigma^2)$.

\begin{enumerate}
	\item[(v)] Apparently this replicated SNCP was originally studied by
	Mal{\'e}cot, see the discussion and references in Felsenstein's paper
	\cite{felsenstein:75} where the following three conditions are
	stated: 
	\begin{itemize}[]
		\item[(I)] ''individuals are distributed randomly on the line with
		equal expected density everywhere''; 
		\item[(II)] ''each individual reproduces
		independently, the number of offspring being drawn from a Poisson
		distribution with a mean of one''; and
		\item[(III)] ''each offspring migrates
		independently, the displacements being drawn from some distribution
		m(x), which we will take to be a normal distribution.''
	\end{itemize}
   (In our notation, $d=1$, $\beta=1$, and $f\sim N_1(\sigma^2)$, but
	\cite{felsenstein:75} considered also more general offspring
	densities $f$ and the cases $d=2,3$.)  \cite{felsenstein:75} noted
	that ``(I) is incompatible with (II)--(III)'' because
	$G_1,G_2,\ldots$ are not stationary Poisson processes and ``a model
	embodying (II) and (III) will lead to the formation of larger and
	larger clumps of individuals separated by greater and greater
	distances'', and then he concluded ``This model is therefore
	biologically irrelevant''.
	
	\item[(vi)] Kingman in \cite{Kingman:77} considered the case where $\beta$ is
	replaced by a non-negative function $b$ which is allowed to depend
	on the cluster centre $x$ and the previous generation,
	% furthermore he extends \eqref{e:SNCP}
	so a cluster with centre $x$ is a Poisson process with intensity
	function $b(x,G_{n-1})f(\nularg-x)$; e.g., as in the Voronoi case
	discussed in item (iv), $b(x,G_{n-1})$ may depend on $G_{n-1}$ in a
	neighbourhood of $x$. Then $G_n$ is a Cox process: $G_n$ conditional
	on $G_{n-1}$ is a Poisson process with intensity function
	\begin{equation}\label{e:Kingmanint}
		\Lambda_n(x)=\sum_{y\in G_{n-1}}b(y,G_{n-1})f(x-y),\qquad x\in\mathbb R^d.
	\end{equation}
	In this setting \cite{Kingman:77} verified that it is impossible for
	$G_n$ to be a stationary Poisson process, however, replacing
	$f(x-y)$ in \eqref{e:Kingmanint} by a more general density which may
	depend on $G_{n-1}-x$, \cite{Kingman:77} noticed that it is possible
	for $G_n$ to be a stationary Poisson process. A trivial example is
	the Voronoi case in item (iv) when $G_n=Z_n$ for $n\ge1$.
\end{enumerate}

Recently, Shimantani in \cite{shimatani:10} considered first the case of items
(a)--(b) and no noise, when $d=2$ and there is the same reproduction
system so that $f_n=f$, $\beta_n=\beta>0$, $\nu_n=\nu$, and
$p_n=p\in(0,1]$ do not depend on $n\ge1$.

\begin{enumerate}
	\item[(vii)] In particular, \cite{shimatani:10} considered the case
	$f\sim N_2(\sigma^2)$ and when $\beta p=1$ or equivalently when the
	intensities $\rho_{G_0}=\rho_{G_1}=\ldots$ are invariant over
	generations, and then he showed that as $n\rightarrow\infty$, the
	PCF for $G_n$ diverges. It
	follows from item (iii) that the model is equivalent to a replicated
	Neyman-Scott process; this becomes a replicated Thomas process when
	each cluster size is Poisson distributed, and hence the 
	 result in \cite{shimatani:10} agrees with the results in
	\cite{felsenstein:75} and \cite{Kingman:77}. Note that
	\cite{shimatani:10} implicitly assumed that a cluster can have more
	than one point. Otherwise the PCF of $G_n$ becomes equal to 1; we
	discuss this rather trivial case again in Section~\ref{s:PCFResults}
	and \ref{s:samesystem}; see also Section~3 in \cite{Kingman:77}.
	% For example, if $G_0$ is a stationary Poisson process and any
	% cluster $Y_{1,x}$ ($x\in G_0$) always consists of one point, it is
	% well-known that $W_1$ is a stationary Poisson process, and so
	% $G_1$ is a stationary Poisson process, and hence by induction any
	% $G_n$ is a stationary Poisson process. See also Section~3 in
	% \cite{Kingman:77}.
\end{enumerate}

Then, Shimantani in \cite{shimatani:10} extended the model by including noise as in
item (d) and by making the following assumptions: The noise processes
$Z_n$ are stationary Poisson processes, satisfying
$0<\rho_{Z_1}=\rho_{Z_2}=\ldots$ and $\rho_{G_0}=\rho_{G_1}=\ldots,$
meaning that $\beta p\le 1$. As there is no noise if $\beta p=1$ it is
also assumed that $\beta p<1$.

\begin{enumerate}
	\item[(viii)] Then \cite{shimatani:10} showed that the PCF of $G_n$
	converges uniformly as $n\rightarrow\infty$ and he argued that this
	limiting case may be ``biologically valid''
	\cite[Section~2.4]{shimatani:10}. However, we address some points arising from \cite{shimatani:10}.
	\begin{itemize}
		\item He did not show that there exists an underlying point
		process having this limiting case as its PCF, although he claimed
		that ``this modified replicated Neyman-Scott process reaches an
		equilibrium state''. In Section~\ref{s:samesystem}, for our more
		general model, we prove the existence of such an underlying point
		process.
		
		\item When $G_0$ is not a stationary Poisson process but its PCF is
		of a particular form (which we specify later in connection to
		\eqref{e:PCFX}), he did not argue that there exists an underlying
		point process and what it could be.  In Section~\ref{s:moments},
		we verify this existence under our more general model.
	\end{itemize}
	% it is not argued in \cite{shimatani:10} whether there exists an
	% underlying point process having this limiting case as its PCF,
	% although \cite{shimatani:10} claimed that "this modified
	% replicated Neyman-Scott process reaches an equilibrium
	% state". Indeed, as we shall verify Furthermore,
	% \cite{shimatani:10} considered the case where $G_0$ is not
	% necessarily a stationary Poisson process but its PCF is a special
	% case of our more general expression in
	% \eqref{e:PCFX0} which is to be discussed later (in contrast to our
	% exposition, \cite{shimatani:10} is not clarifying that apart from
	% the stationary Poisson point case there do exist stationary point
	% processes with such a PCF).
\end{enumerate}

Finally, we remark on a few related cases.

\begin{enumerate}
	\item[(ix)] Whilst we study the processes $G_n$ for all
	$n=1,2,\ldots$, often in the spatial point process literature the
	focus is on either $G_1$ or $G_2$, assuming $p_n=1$ and
	$\rho_{Z_n}=0$ for $n=1$ or $n=1,2$,
	respectively. \cite{wiegandetal:07} studied this in the special case
	of a double Thomas cluster process $G_2$ when $d=2$, i.e., when
	$G_0$ is a stationary Poisson process, \eqref{e:SNCP} holds for both
	$G_1=Y_1$ and $G_2=Y_2$, and $f_n\sim N_2(\sigma_n^2)$ for $n=1,2$;
	see also \cite{Andersen:etal:17} for more general functions
	$f_n$. Moreover, \cite{wiegandetal:07} extended the double Thomas
	process to the case where $\rho_{Z_1}=0$ and $\rho_{Z_2}>0$; this
	type of model is also considered in \cite{Andersen:etal:17}.  In any
	case, our general results for intensities and PCFs in
	Section~\ref{s:moments} will cover all these cases.
	
	\item[(x)] If for each generation we assume no thinning ($p_1=p_2=\ldots=1$), no noise ($\rho_{Z_1}=\rho_{Z_2}=\ldots=0$), no retention ($q_1=q_2=\ldots=0$) as well as $\beta_1=\beta_2=\ldots$ and
	$f_1=f_2=\ldots$, then the superposition $\bigcup_{n=0}^\infty G_n$ is
	known as a spatial Hawkes process, see \cite{moeller:torrisi:07} and
	the references therein.
\end{enumerate}

\section{First and second order moment properties}\label{s:moments}

In this section we determine the intensity and the PCF of $G_n$ for
$n=1,2,\ldots,$ under more general assumptions than in Shimatani's paper 
\cite{shimatani:10}. Specifically, points from one generation can be retained in the next generation, the noise is an arbitrary
stationary point process (not necessarily a stationary Poisson process
as in \cite{shimatani:10}), and we do not assume the same reproduction
system.

\subsection{Intensities}
By induction $G_n$ is seen to be stationary for $n=0,1,\ldots$ Its
intensity is determined in the following proposition where for
notational convenience we define $Z_0=G_0$ so that
$\rho_{Z_0}=\rho_{G_0}$.

\begin{proposition}\label{p:intensity}
	For $n=1,2,\ldots,$ we have that $G_n$ is stationary with a positive
	and finite intensity given by
	\begin{align}\label{e:intensityXn}
		\rho_{G_n} 
		= \rho_{G_{n-1}}(\beta_n p_n + q_n) + \rho_{Z_n}
		=\rho_{Z_n} + \sum_{i=0}^{n-1}\rho_{Z_i}\prod_{j=i+1}^n(\beta_jp_j + q_j).
	\end{align}
\end{proposition}

\begin{proof} 
	Using induction for $n=1,2,\ldots,$ the proposition follows
	immediately from items (a)--(d), where the term
	$\rho_{Z_i}\prod_{j=i+1}^n(\beta_jp_j+q_j)$ is the contribution to the
	intensity caused by the clusters with centres $Z_i$ and after applying the two types of
	independent thinnings.
\end{proof}

\subsection{Pair correlation functions}

\subsubsection{Preliminaries}\label{s:preliminaries}
Recall that a stationary point process $X\subset\mathbb R^d$ with
intensity $\rho_X\in(0,\infty)$ has a translation invariant PCF (pair
correlation function) $(u,v)\to g_X(u-v)$ with
$(u,v)\in\mathbb R^d\times\mathbb R^d$ if for any bounded Borel
function $h: \mathbb R^d\times\mathbb R^d \to [0, \infty)$ with
compact support,
\begin{equation}\label{e:PCF}
	\mathrm{E} \sum_{x_1,x_2 \in X:\,x_1 \neq x_2} h(x_1,x_2) 
	= \rho_X^2\iint h(x_1,x_2)g_X(x_1-x_2)\,\mathrm dx_1\,\mathrm
	dx_2<\infty.
\end{equation}
% where $\sum^{\neq}$ means that $x_1 \neq x_2$.
Equivalently, for any bounded and disjoint Borel sets
$A,B\subset\mathbb R^d$, denoting $N(A)$ the cardinality of $X\cap A$,
the covariance between $N(A)$ and $N(B)$ exists and is given by
\begin{equation*}
	\mathrm{cov}\{N(A),N(B)\}
	=\rho_X^2\int_A\int_B\left\{g_{X}(x_1-x_2)-1\right\}\,\mathrm dx_1\,\mathrm dx_2.
\end{equation*}
% We notice the following.

Some remarks are in order.  Note that $g_X$ is uniquely determined
except for nullsets with respect to Lebesgue measure on $\mathbb R^d$,
but we ignore such nullsets in the following.  Thus the translation
invariance of the PCF is implied by the stationarity of $X$.  Our
results below are presented in terms of the reduced PCF $g_X-1$ rather than $g_X$, and
$g_X=1$ if $X$ is a Poisson process.
% {\sf JM: Nu definerer vi, hvad vi ofte betegner med $\gamma$ - faar
% vi mon brug for dette eller skal det fjenes? - Eller skal vi havde
% en equation-formula, hvor $\gamma$ defineres?}  The integral
% $\int\left(g_{X}-1\right)$ is a rough measure of the amount of
% positive/negative association between the points in $X$, and
% comparing two such measures for two different point processes makes
% only sense if the processes have equal intensities (\cite{LMR15}).
% For simplicity, we abuse terminology and also refer to $g_X$ as the
% PCF.  Intuitively, $g_X(x_1-x_2)$ is the ratio of two mean values:
% the expected number of pairs of points in $X$ with one point in an
% infinitesimally small region around $x_1$ and the other point in an
% infinitesimally small region around $x_2$; and what this mean value
% is for a stationary Poisson process with intensity function
% $\rho_X$.
It is convenient when $g_{X}$ is isotropic, meaning that there is a function $g_{X,o}$ so that for all $x\in\mathbb R^d$,  
$g_{X}(x)=g_{X,o}(\norm{x})$ depends only on $x$ through $\|x\|$. With a slight
abuse of terminology, we also refer to $g_X$ and $g_{X,o}$ as PCFs.

For a PDF $h$ on $\mathbb R^d$, let $\tilde h(x)=h(-x)$ and let
\begin{equation}\label{e:convul}
	h*\tilde h(x_1-x_2)=\int h(x_1-y)h(x_2-y)\,\mathrm dy
\end{equation} 
be the convolution of $h$ and $\tilde h$. Note that if $U$ and $V$ are
IID random variables with PDF $h$, then $U-V$ has PDF $h*\tilde h$.
In the following section we consider the case
\begin{equation}\label{e:PCFX}
	g_X - 1 = a~h*\tilde{h}
\end{equation}
for real constants $a$, where $X$ in particular, may refer to the
initial generation process, $G_0$, or the noise process, $Z_n$.  This
corresponds to $X$ being a Poisson process if $a=0$, a point process
with positive association between its points (attractiveness,
clustering, or clumping) if $a>0$, and a point process with negative
association between its points (repulsiveness or regularity) if $a<0$.
In \cite{shimatani:10}, for the initial generation process $G_0$, Shimatani
briefly discussed the special case of \eqref{e:PCFX} when
$h\sim N_2(\tau^2/2)$ (so $h*\tilde{h}\sim N_2(\tau^2)$) whilst the
noise processses are stationary Poisson processes. However, if $a\ne0$
he did not argue if an underlying point process with PCF $g_X$
exists. Indeed, as detailed in Appendix~\ref{a:1}, there exist
$\alpha$-weighted determinantal point processes satisfying
\eqref{e:PCFX} if $\alpha=-1/a$ is a positive integer, and there exist
Cox processes given by $\alpha$-weighted permanental point processes
satisfying \eqref{e:PCFX} if $\alpha=1/a$ is a positive
half-integer.  
Additionally, $h$ needs not to be Gaussian when dealing
with weighted determinantal and permanental point processes; e.g.\ $h$
may be the density of a normal-variance mixture distribution
\cite{Barndorff:etal:82}.  Also generalized shot-noise Cox processes \cite{moeller:torrisi:07} have PCFs of the form \eqref{e:PCFX} with $a>0$. Moreover, \eqref{e:PCFX} holds for
many other cases of point process models for $X$:
% Let $\mathcal F$ and $\mathcal F^{-1}$ denote Fourier and inverse
% Fourier transform, respectively.
If the Fourier transform $\mathcal F\left(g_{X}-1\right)$ is
well-defined and non-negative,
% then $g_{G_0}-1$ is an auto-covariance function (that is, the
% function
% $\mathbb R^d\times\mathbb R^d\ni (u,v)\to g_{G_0}(u,v)-1$ is
% positive semi-definite), and so
if $h=\tilde{h}$, and if $a:=\int(g_X-1)\in(0,\infty)$, then
\eqref{e:PCFX} holds with
\begin{equation*}
	h  =\mathcal F^{-1}\bigl\{\sqrt{\mathcal
		F(g_X-1)}\bigr\}\big/\sqrt{a}
\end{equation*}
provided this inverse transform is well-defined.
Extensions of \ref{e:PCFX} are discussed in Section~\ref{s:extention}

We will need the following lemma in Section~\ref{s:PCFResultsProof}.
\begin{lemma}\label{l:2}
	Suppose $g_X$ is of the form \eqref{e:PCFX}. Then 
	\begin{equation*}\label{e:01case}
		\iint \{g_X(x_1-x_2)-1\}f(u-x_1)f(v-x_2)\,\mathrm dx_1\,\mathrm dx_2
		= ah*\tilde{h}*f*\tilde f(u-v)
	\end{equation*}
	for any integrable real function
	$f$ defined on $\mathbb R^d$ and for any $u,v\in\mathbb R^d$.
\end{lemma}
\begin{proof}
	Follows from \eqref{e:convul} and \eqref{e:PCFX} using Fubini's
	theorem and the fact that the convolution operation is commutative
	and associative.
\end{proof}

\subsubsection{First main result}\label{s:PCFResults}

This section concerns our first main result, Theorem~\ref{t:main}, which is verified in Section~\eqref{s:PCFResultsProof}.  We
use the following notation.
% Denote $g_{n}:=g_{G_{n}}$ the PCF of $G_{n}$ provided it exists.
Define
\begin{equation}\label{e:cndef}
	c_n 
	= \mathrm E\{\#Y_{n,x}(\#Y_{n,x}-1)\}/\beta_n^2
	= (\nu_n+\beta_n^2-\beta_n)/\beta_n^2\qquad\text{if $\beta_n>0$},
\end{equation}
with $c_n=0$ if $\beta_n=0$.  If $\beta_n=\nu_n>0$, as in the case
when $\#Y_{n,x}$ follows a (non-degenerated) Poisson distribution,
then $c_n=1$. The case of overdispersion (underdispersion), that is,
$\nu_n > \beta_n$ ($\nu_n < \beta_n$) corresponds to $c_n>1$
($c_n<1$).
  Denote by $\delta_0$ the Dirac delta function defined on $\mathbb R^d$. Recall that for any integrable real function $h$  defined on $\mathbb{R}^d$, $h*\delta_0=\delta_0*h=h$, and for any $a\in\mathbb R$, $a\delta_0*a\delta_0=a^2\delta_0$, where we understand $0\delta_0$ as 0. Finally, let 
  $\Ast_{i=1}^{n}h_i = h_1*\cdots*h_n$ 
  where each $h_i$ is either of the form $a_i\delta_0$, with $a_i$ a real constant, or it is an integrable real function defined on $\mathbb R^d$.

\begin{theorem}\label{t:main}
	Suppose $g_{G_0}$ and $g_{G_{Z_n}}$ are of the form \eqref{e:PCFX},
	that is, $g_{G_0}-1=af_0*\tilde{f}_0$ and
	$g_{Z_n}-1=b_nf_{Z_n}*\tilde{f}_{Z_n}$ for $n=1,2,\ldots$. Then, for
	all $u\in\mathbb R^d$ and $n=1,2,\ldots$,
%	\com{\begin{align}
%		\MoveEqLeft[3] g_{G_n}(u)-1 = {}
%		\biggl(\frac{\rho_{G_0}}{\rho_{G_n}}\biggr)^2af_0*\tilde{f}_0(u)*\overset{n}{\underset{i=1}{\Ast}}\left\{\beta_i^2p_i^2f_i*\tilde{f}_i(u) + q_i^2\right\} \label{e:main1}
%		\\
%		&+\sum_{i=1}^n\frac{c_i\rho_{G_{i-1}}}{\rho_{G_n}^2}~\overset{n}{\underset{j=i}{\Ast}}\left\{\beta_j^2p_j^2f_j*\tilde{f}_j(u) + q_j^2\mathbb{I}(i\ne j)\right\} \label{e:main2}
%		\\
%		&+\sum_{i=1}^{n-1}\biggl(\frac{\rho_{Z_i}}{\rho_{G_n}}\biggr)^2 b_if_{Z_i}*\tilde{f}_{Z_i}(u)*\overset{n}{\underset{j=i+1}{\Ast}}\left\{\beta_j^2p_j^2f_j*\tilde{f}_j(u) + q_j^2\right\} \label{e:main3}
%		\\
%		&+\left(\frac{\rho_{Z_n}}{\rho_{G_n}}\right)^2b_nf_{Z_n}*\tilde{f}_{Z_n}(u)\label{e:main4}
%		\end{align}}
%		{
		\begin{align}
			g_{G_n}(u)-1
			={}&\biggl(\frac{\rho_{G_0}}{\rho_{G_n}}\biggr)^2af_0*\tilde{f}_0*\overset{n}{\underset{i=1}{\Ast}}\bigg\{\big(\beta_ip_if_i + q_i\delta_0\big)*\big(\beta_ip_i\tilde{f}_i + q_i\delta_0\big)\bigg\}(u) \label{e:main1} \\
			&\begin{aligned}
				&+\sum_{i=1}^n\frac{\rho_{G_{i-1}}}{\rho_{G_n}^2}\bigg\{c_i\big(\beta_ip_i\big)^2f_i*\tilde{f}_i + \beta_ip_iq_i\big(f_i+\tilde f_i\big)\bigg\} \\
				&\hspace{13mm}*\overset{n}{\underset{j=i+1}{\Ast}}\bigg\{\big(\beta_jp_jf_j + q_j\delta_0\big)*\big(\beta_jp_j\tilde{f}_j + q_j\delta_0\big)\bigg\}(u) 
			\end{aligned}\label{e:main2}\\
			&\begin{aligned}
				&+\sum_{i=1}^{n-1}\biggl(\frac{\rho_{Z_i}}{\rho_{G_n}}\biggr)^2 b_if_{Z_i}*\tilde{f}_{Z_i}\\
				&\hspace{13mm}*\overset{n}{\underset{j=i+1}{\Ast}}\bigg\{\big(\beta_jp_jf_j + q_j\delta_0\big)*\big(\beta_jp_j\tilde{f}_j + q_j\delta_0\big)\bigg\}(u)
			\end{aligned} \label{e:main3}\\
			&+\left(\frac{\rho_{Z_n}}{\rho_{G_n}}\right)^2b_nf_{Z_n}*\tilde{f}_{Z_n}(u)\label{e:main4}
		\end{align} %}
		where the $(n-i)$-th fold convolution in \eqref{e:main2} is interpreted as 1 if $i=n$ %$\mathbb{I}(\cdot)$ is the indicator function, 
		and the sum in \eqref{e:main3} is interpreted as 0 if $n=1$. %\com{and
%		\begin{equation*}
%		\overset{n}{\underset{i=1}{\Ast}}\left(h_i + a_i\right)
%		=\sum_{k=0}^{n}\sum_{1\le i_1<\dots<i_k\le n}h_{i_1}*\cdots*h_{i_k}\prod_{j\in\left\{1,\ldots, n\right\}\setminus\left\{i_1, \ldots, i_k\right\}}a_j,
%		\end{equation*}
%		where the term for $k=0$ is interpreted as $\prod_{j=1}^n a_j$.}{
		 %}
\end{theorem}

The terms in \eqref{e:main1}--\eqref{e:main4} have the following
interpretation when $u\ne0$: The right side of \eqref{e:main1} corresponds to
pairs of $n$-th generation points with different $0$-th generation
ancestors; the $i$-th term in \eqref{e:main2} corresponds to pairs of $n$-th generation points 
when they have a common $(i-1)$-th generation ancestor initiated by $Z_{i-1}$ if $i>1$ or by $G_0$ if $i=1$; the
$i$-th term in \eqref{e:main3} corresponds to pairs of $n$-th
generation points with different $i$-th generation ancestors initiated
by  $Z_i$; and %the term in 
\eqref{e:main4}
corresponds to point pairs in $Z_n$.

% The terms in \eqref{e:main2}-\eqref{e:main4} are zero when
% $b_1=b_2=\ldots=0$ (as is the case if $Z_1,Z_2,\ldots$ are Poisson
% processes) and when $c_1=c_2=\ldots=0$.

Later in Section~\ref{s:second main}, our main interest is in the
behaviour of $g_{G_n}$ as $n\rightarrow\infty$ when we have the same
reproduction system, but for the moment, it is worth noticing the
flexibility of our model for $G_1$ and the effect of the choice of its
centre process $G_0$: For simplicity, suppose there is no noise and no retention of points from one generation to the next (i.e., $\rho_{Z_n}=q_n=0$ for $n=1,2,\ldots$), and $G_0$ is
stationary and either a Poisson or an weighted determinantal or
permanental point process with a Gaussian kernel. Specifically, suppose
$d=\nobreak 2$, $G_0$ has intensity $\rho_{G_0}=100$, and using a
notation as in Appendix~\ref{a:1}, the Gaussian kernel has an
auto-correlation function of the form $R(x)=\exp(-\|x/\tau\|^2)$,
where the value of $\tau$ depends on the type of process: For the
$\alpha$-weighted determinantal point process, we consider the most
repulsive case, that is, a determinantal point process ($\alpha=1$)
and $\tau=1/\sqrt{\rho_{G_0}\pi}$ is largest possible to ensure
existence of the process \cite{LMR15}; for the $\alpha$-weighted
permanental point process, $\alpha=1/2$ (the most attractive case when
it is also a Cox process, see Appendix~\ref{a:1}) and $\tau=0.1$ is an
arbitrary value (any positive number can be used).  Note that
$R^2=(\sqrt{\pi}\tau)^2f_0*\tilde{f}_0$ where $f_0\sim N_2(\tau^2/8)$,
which by \eqref{e:pcfwPPP} and \eqref{e:pcfwDPP} mean that
\eqref{e:PCFX} is satisfied with $a=2(\sqrt{\pi}\tau)^2$ and
$a=-(\sqrt{\pi}\tau)^2$ for the weighted permanental and determinantal
point processes, respectively, and $a=0$ in case of the Poisson
process.  Moreover, let the number of points in a cluster be Poisson
distributed with mean $\beta_1=10$, $p_1=1$, and
$f_1\sim N_2(\sigma^2)$, with $\sigma=0.01$.  Then, by
Theorem~\ref{t:main},
\begin{align*}
	g_{G_1}(u)-1 ={} & \frac{a}{2\pi(2\sigma^2 +
			\tau^2/4)}\exp\left\{-\frac{\norm{u}^2}{2(2\sigma^2 +
		\tau^2/4)}\right\}
	\\
	&+\frac{1}{\rho_{G_0}4\pi\sigma^2}\exp\left(-\frac{\norm{u}^2}{4\sigma^2}\right).
\end{align*}
In Figure~\ref{f:pcfs} we present the isotropic PCF $g_{G_1,o}(r)=g_{G_1}(u)$
% (using a notation as in Section~\ref{s:preliminaries})
as a function of the inter-point distance $r=\|u\|$ in case of each of
the three models of $G_0$, where using an obvious notation,
$g_{G_1,o}^{\textup{det}}<g_{G_1,o}^{\textup{Pois}}<g_{G_1,o}^{\text{wper}}$.
Most notable is the fact that $g_{G_1,o}^{\textup{det}}(r)$ exhibits
repulsion at midrange distances $r$. For $g_{G_1,o}^{\textup{wper}}$, we
see a high degree of clustering, which is persistent for large values
of $r$; this will of course be even more pronounced if we increase the
value of $\tau$; %$a^{\text{wper}}$;
whilst decreasing $\sigma$ will increase the peak at small values
of~$r$.  Figure~\ref{f:sim} shows simulations of $G_1$ in each of the
three cases of the model of $G_0$.  As expected, we clearly see a
higher degree of repulsion when $G_0$ is a determinantal point process
(the left most plot) and a higher degree of clustering when $G_0$ is a
weighted permanental point process (the right most plot). In
particular, the clusters are more distinguishable when $G_0$ is a
determinantal point process, and this will be even more pronounced if
decreasing $\sigma$ because the spread of clusters then decrease.

\begin{figure}[ht!]
	\centering
	\includegraphics[scale=0.33]{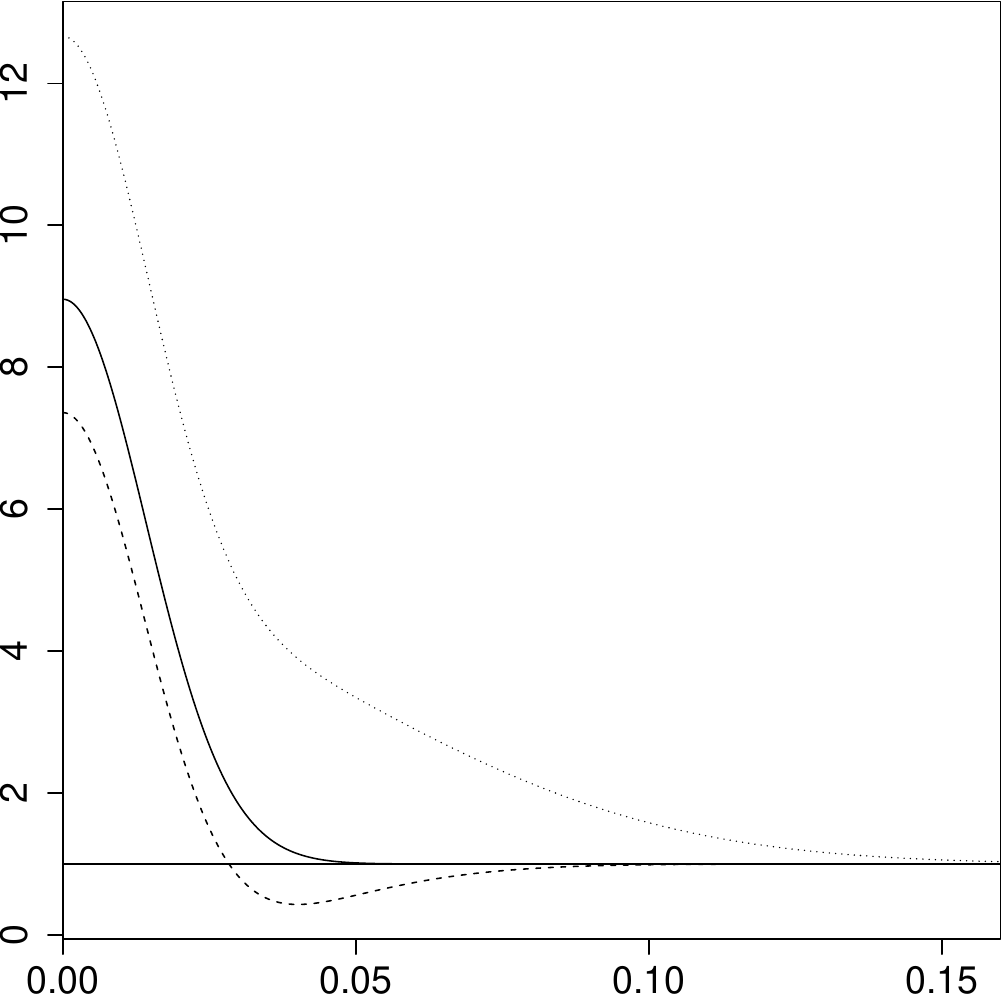}
	\caption{The PCFs of $G_1$ when $G_0$ is a determinantal, Poisson,
		or weighted permanental point process (dashed, solid, and dotted
		respectively), with parameters and Gaussian offspring PDF as
		specified in the text. The solid horizontal line is the
		 PCF for a Poisson process.}
	\label{f:pcfs}
\end{figure}
\begin{figure}[ht!]
	\centering
	\includegraphics[width=0.31\textwidth]{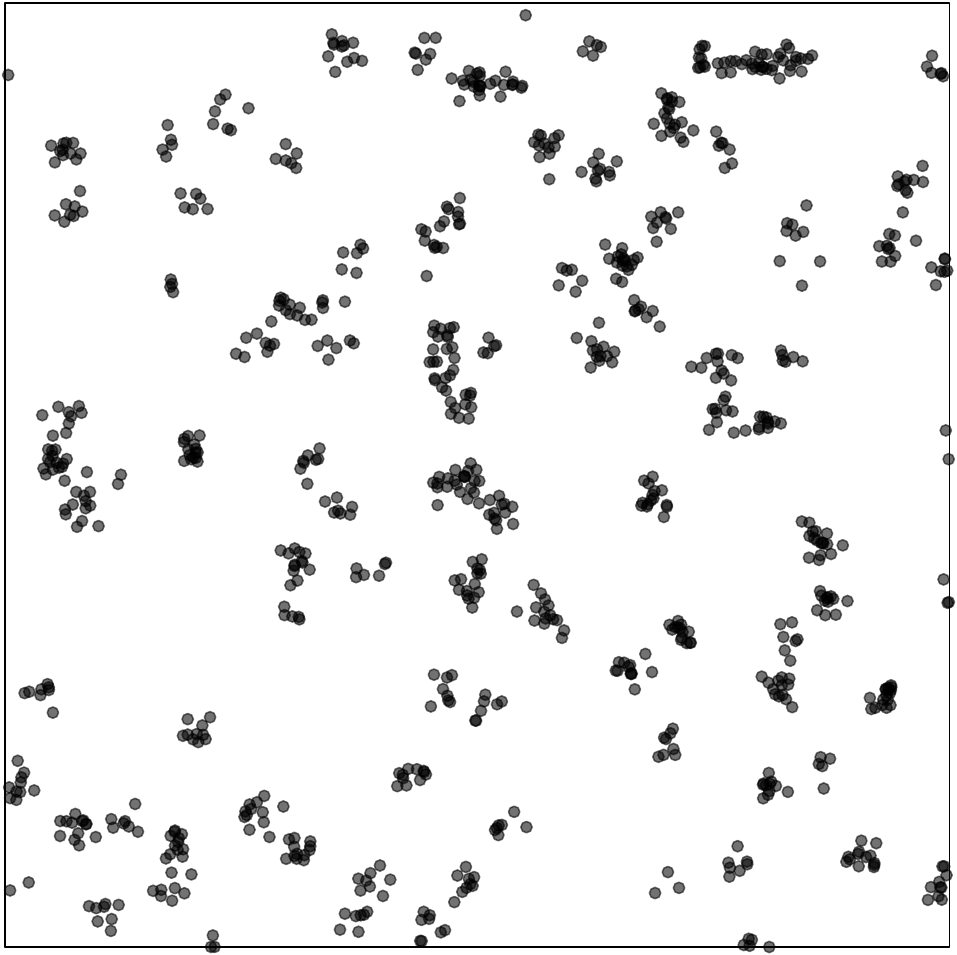}
	\hfill
	\includegraphics[width=0.31\textwidth]{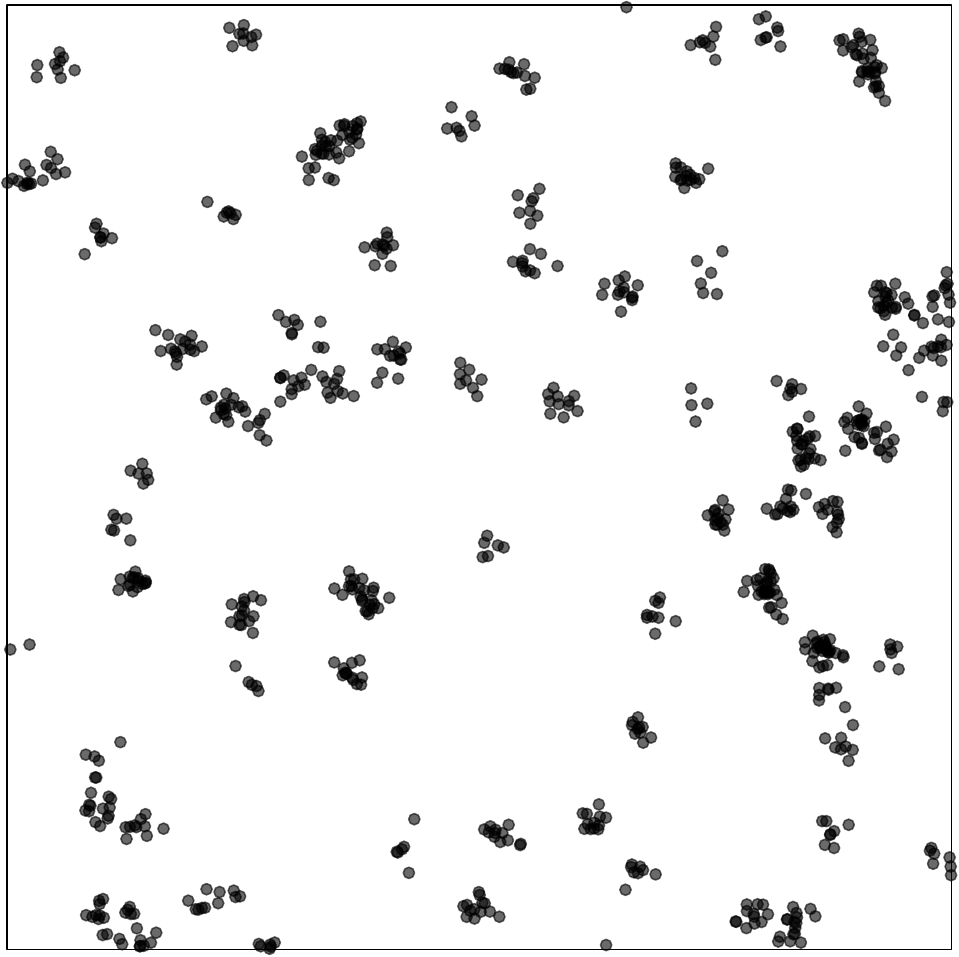}
	\hfill
	\includegraphics[width=0.31\textwidth]{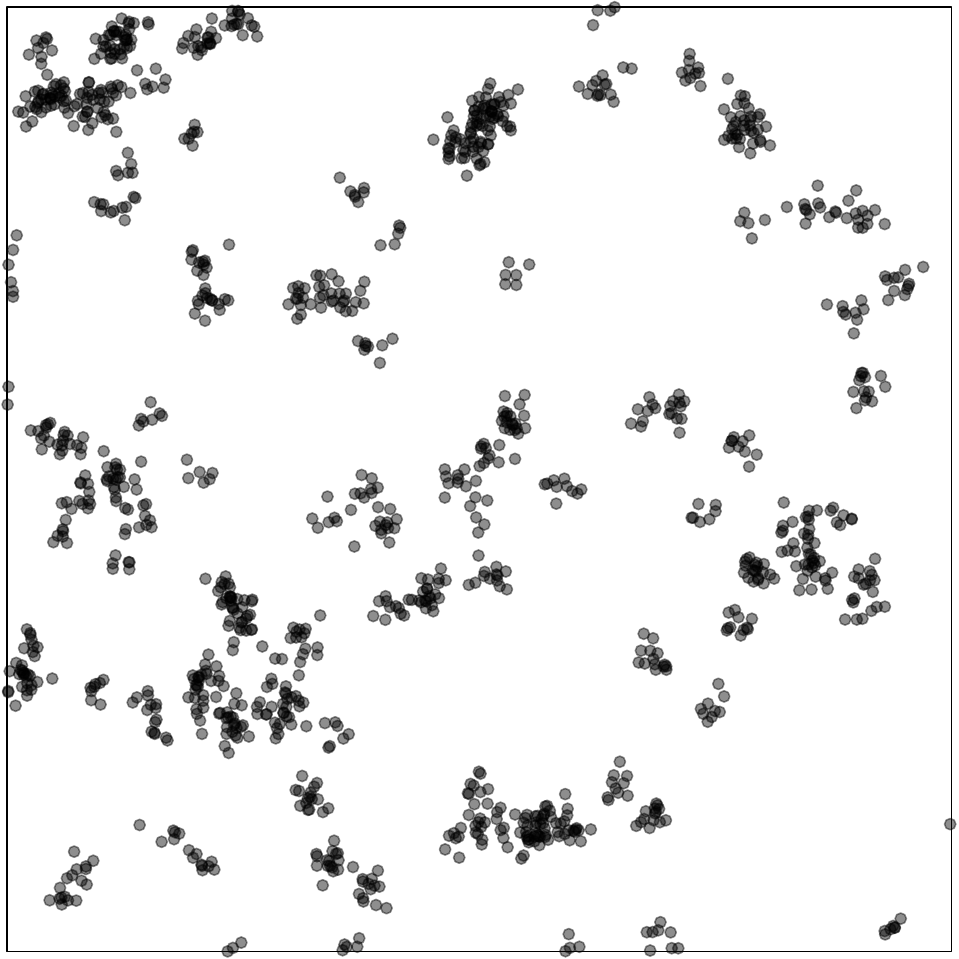}
	\caption{Simulations of $G_1$ restricted to a unit square when $G_0$
		is a determinantal (left panel), Poisson (middle panel), or
		weighted permanental (right panel) point process, see
		Figure~\ref{f:pcfs} and the text.}
	\label{f:sim}
\end{figure}

\subsubsection{Proof of Theorem~\ref{t:main}}\label{s:PCFResultsProof}
Shimatani in \cite{shimatani:10} verified Theorem~\ref{t:main} in the special case where $q_1=q_2=\dots=0$, 
$b_1=b_2=\dots=0$ (as is the case if $Z_1,Z_2,\ldots$ are stationary
Poisson processes), and $c_1=c_2=\dots>0$, in which case the terms in
\eqref{e:main3}--\eqref{e:main4} are zero. If $c_1=c_2=\dots=0$, then
\eqref{e:main2} is zero and by \eqref{e:cndef}, with probability one,
$\#Y_{n,x}\in\{0,1\}$ for all $x\in G_{n-1}$ and $n=1,2,\ldots$.
Consequently, the proof of Theorem~\ref{t:main} is trivial if
$c_1=c_2=\dots=0$ and both $G_0$ and $Z_1,Z_2,\ldots$ are stationary
Poisson processes, because then $a=0$, $b_1=b_2=\dots=0$,
$G_1,G_2,\ldots$ are stationary Poisson processes, and the class of stationary Poisson processes is closed under IID random
	shifts of the points, thinning, and superposition.
% implicitly assumed $c_1=c_2=\ldots>0$. Otherwise, with probability
% one, $\#Y_{n,x}\in\{0,1\}$ for all $x\in G_{n-1}$ and
% $n=1,2,\ldots$, cf.\ \eqref{e:cndef}. Thus, if $c_1=c_2=\ldots=0$
% and both $G_0$ and $Z_1,Z_2,\ldots$ are stationary Poisson
% processes, then $G_1,G_2,\ldots$ are stationary Poisson processes
% because IID random shifts of the points in a stationary Poisson
% process generate a stationary Poisson process).
The general proof of Theorem~\ref{t:main} follows by induction from
the following lemma %Lemma~\ref{l:1} 
together with Lemma~\ref{l:2}.

\begin{lemma}\label{l:1} %{e:PCFXn}  
	If $\rho_{G_{n-1}}>0$, $\rho_{G_n}>0$, and $g_{G_{n-1}}$ and
	$g_{Z_{n}}$ exist, then $g_{G_n}$ exists and is given by
%\com{\begin{equation*}\label{e:PCFXn}
%	\begin{aligned}
%		&g_{G_n}(u-v)-1 \\
%		&=\left(\frac{\rho_{G_{n-1}}\beta_np_n}{\rho_{G_n}}\right)^2
%		\biggl[\iint \{g_{G_{n-1}}(x_1-x_2)-1\}f_n(u-x_1)
%		f_n(v-x_2)\,\mathrm dx_1\,\mathrm dx_2 \\
%		&\hspace{3cm}+\frac{c_n}{\rho_{G_{n-1}}}f_n*\tilde f_n(u-v)\biggr] \\
%		&+\left(\frac{\rho_{G_{n-1}}q_n}{\rho_{G_n}}\right)^2\left\{g_{G_{n-1}}(u-v) - 1\right\} +\left(\frac{\rho_{Z_n}}{\rho_{G_n}}\right)^2\left\{g_{Z_n}(u-v)-1\right\}
%	\end{aligned}
%\end{equation*}}
\begin{align*}\label{e:PCFXn}
		\MoveEqLeft[2]g_{G_n}(u-v)-1 \\
		&=\left(\frac{\rho_{G_{n-1}}\beta_np_n}{\rho_{G_n}}\right)^2
		\biggl[\iint \{g_{G_{n-1}}(x_1-x_2)-1\}f_n(u-x_1)
		f_n(v-x_2)\,\mathrm dx_1\,\mathrm dx_2 \\
		&\hspace{35mm}+\frac{c_n}{\rho_{G_{n-1}}}f_n*\tilde f_n(u-v)\biggr] \\
		&+\left(\frac{\rho_{G_{n-1}}q_n}{\rho_{G_n}}\right)^2\left\{g_{G_{n-1}}(u-v) - 1\right\} +\left(\frac{\rho_{Z_n}}{\rho_{G_n}}\right)^2\left\{g_{Z_n}(u-v)-1\right\} \\
		&+\frac{\rho_{G_{n-1}}^2\beta_np_nq_n}{\rho_{G_n}^2}\bigg[\int\left\{g_{G_{n-1}}(v-x)-1\right\}\left\{f_n(u-x)+\tilde f_n(u-x)\right\}\,\mathrm dx\\
		&\hspace{2.8cm}+\frac{1}{\rho_{G_{n-1}}}\left\{f_n(u-v)+\tilde f_n(u-v)\right\}\bigg]
\end{align*}
%	\begin{equation}\label{e:PCFXn}
%		\begin{aligned}
%			\MoveEqLeft[3] g_{G_n}(u-v)-1
%			= \smash[t]{\left(\frac{\rho_{G_{n-1}}\beta_np_n}{\rho_{G_n}}\right)^2}
%			\\
%			&\cdot
%			\begin{aligned}[t]
%				\biggl[&\iint \{g_{G_{n-1}}(x_1-x_2)-1\}f_n(u-x_1)\tilde
%				f_n(v-x_2)\,\mathrm dx_1\,\mathrm dx_2
%				\\
%				&+\frac{c_n}{\rho_{G_{n-1}}}f_n*\tilde f_n(u-v)\biggr] +
%				\left(\frac{\rho_{Z_n}}{\rho_{G_n}}\right)^2\left\{g_{Z_n}(u-v)-1\right\}
%			\end{aligned}
%		\end{aligned}
%	\end{equation}
	for any $u,v\in\mathbb R^d$.
\end{lemma}

\goodbreak

\begin{proof}
	Note that $Y_n$ is stationary with intensity
	\begin{equation}\label{e:hhh1}
		\rho_{Y_n}=\rho_{G_{n-1}}\beta_n.
	\end{equation} 
	It follows straightforwardly from \eqref{e:DefYn}, \eqref{e:PCF},
	and Fubini's theorem that its PCF is given by
	\begin{equation}\label{e:hhh2}
		\begin{aligned}
			\rho_{Y_n}^2g_{Y_n}(u-v)
			= {} & \rho^2_{G_{n-1}}\beta_n^2\iint g_{G_{n-1}}(x_1-x_2)f_n(u-x_1)f_n(v-x_2)\,\mathrm dx_1\,\mathrm dx_2 \\
			&+ \rho_{G_{n-1}}c_n \beta_n^2 f_n*\tilde f_n(u-v)
		\end{aligned}
	\end{equation} 
	for any $u,v\in\mathbb R^d$, where the two terms on the right hand
	side correspond to pairs of points from $Y_n$ belonging to different
	clusters and the same cluster, respectively.  Hence by
	\eqref{e:DefWn} and \eqref{e:hhh1}, $W_n$ is stationary with
	intensity
	\begin{equation}\label{e:hhh3}
		\rho_{W_n}=p_n\rho_{Y_n}=\rho_{G_{n-1}}\beta_n p_n
	\end{equation}
	and PCF
	\begin{equation}\label{e:hhh4}
		\begin{aligned}
			g_{W_n}(u-v) ={} &g_{Y_n}(u-v)
			\\
			= {} &
			\iint g_ {G_{n-1}}(x_1-x_2)f_n(u-x_1)f_n(v-x_2)\,\mathrm dx_1\,\mathrm dx_2\\
			&+\frac{c_n}{\rho_{G_{n-1}}}f_n*\tilde f_n(u-v)
		\end{aligned}
	\end{equation}
	where the first identify follows from the fact that PCFs are
	invariant under independent thinning, and where \eqref{e:hhh2} is
	used to obtain the second identity.  
	%\com{Furthermore, it follows straightforwardly from \eqref{e:DefXn}, \eqref{e:PCF}, and Fubini's
		%theorem}{
		Also, for disjoint Borel sets $A_1,A_2\subseteq\mathbb R^d$, it follows from items (a)--(c) that 
	\begin{align*}
		\MoveEqLeft[33]\mathrm E\{\#(W_n\cap A_1)\#(G_{n-1}^{\rm thin}\cap A_2)\} \\ 
		= \rho_{W_n}\rho_{G_{n-1}^{\rm thin}}\int_{A_1}\int_{A_2} \left\{ \frac{1}{\rho_{G_{n-1}}}f_n(x_1-x_2) + g_{G_{n-1}}*\tilde{f}_n(x_1-x_2)\right\}\,\mathrm dx_1\mathrm dx_2.
	\end{align*}
	Furthermore, by \eqref{e:DefXn}, \eqref{e:PCF}, and Fubini's
	theorem it is readily seen
	that $G_n$ has PCF given by
%\com{\begin{equation*}
%		\begin{aligned}
%			\rho_{G_n}^2g_{G_n}(x)
%			= {} &\rho_{W_n}^2g_{W_n}(x) + \rho_{Z_n}^2g_{Z_n}(x) + \left(\rho_{G_{n-1}}^{\rm thin}\right)^2g_{G_{n-1}^{\rm thin}}(x) \\
%			&+ 2\rho_{W_n}\rho_{Z_n} + 2\rho_{Z_n}\rho_{G_{n-1}}^{\rm thin} + 2\rho_{W_n}\rho_{G_{n-1}}^{\rm thin}
%		\end{aligned}
%	\end{equation*}}{
	\begin{equation*}
		\begin{aligned}
			\rho_{G_n}^2g_{G_n}(x)
			= {} &\rho_{W_n}^2g_{W_n}(x) + \left(\rho_{G_{n-1}}^{\rm thin}\right)^2g_{G_{n-1}^{\rm thin}}(x) + \rho_{Z_n}^2g_{Z_n}(x) \\
			&+ 2\rho_{W_n}\rho_{Z_n} + 2\rho_{Z_n}\rho_{G_{n-1}}^{\rm thin} \\
			&+ 2\rho_{W_n}\rho_{G_{n-1}^{\rm thin}}
			\left\{g_{G_{n-1}}*f_n(x) + \frac{1}{\rho_{G_{n-1}}}f_n(x)\right\}
		\end{aligned}
	\end{equation*} %}
	where the six terms on the right hand side correspond to pairs of
		points from $W_n$, $Z_n$, $G_{n-1}^{\rm thin}$, $W_n$ and $Z_n$, $Z_n$ and $G_{n-1}^{\rm thin}$, and $W_n$ and $G_{n-1}^{\rm thin}$, respectively, where the latter three cases can be ordered in two
		ways. Combining all this with the first
	identity in \eqref{e:intensityXn} and \eqref{e:hhh3}, we easily
	obtain
%	\com{\begin{equation*}
%		\begin{aligned}
%		g_{G_n}(x)-1
%		= {} &\left(\frac{\rho_{G_{n-1}}\beta_np_n}{\rho_{G_n}}\right)^2\left\{g_{W_n}(x)-1\right\}
%		+\left(\frac{\rho_{Z_n}}{\rho_{G_n}}\right)^2\left\{g_{Z_n}(x)-1\right\} \\
%		&+\left(\frac{\rho_{G_{n-1}}q_n}{\rho_{G_n}}\right)^2\left\{g_{G_{n-1}}(x) - 1\right\}
%		\end{aligned}
%	\end{equation*}}{
	\begin{equation*}
		\begin{aligned}
		\MoveEqLeft[3] g_{G_n}(u-v)-1 \\
		= {} &\left(\frac{\rho_{G_{n-1}}\beta_np_n}{\rho_{G_n}}\right)^2\left\{g_{W_n}(u-v)-1\right\}
		+\left(\frac{\rho_{Z_n}}{\rho_{G_n}}\right)^2\left\{g_{Z_n}(u-v)-1\right\} \\
		&+\left(\frac{\rho_{G_{n-1}}q_n}{\rho_{G_n}}\right)^2\left\{g_{G_{n-1}}(u-v) - 1\right\} \\
		&+ \rho_{W_n}\rho_{G_{n-1}^{\rm thin}}
		\bigg[\int (g_{G_{n-1}}(v-x) - 1)\left\{f_n(u - x)+\tilde f_n(u-x)\right\}\, \mathrm dx \\
		&\hspace{3cm}+ \frac{1}{\rho_{G_{n-1}}}\left\{f_n(u - v)+\tilde f_n(u-v)\right\}\bigg].
		\end{aligned}
	\end{equation*} %}
	This combined with \eqref{e:hhh4} imply the result in Lemma~\ref{l:1}. %\eqref{e:PCFXn}.
\end{proof}

\subsubsection{Extension}\label{s:extention}
More generally than in Section~\ref{s:PCFResults} we may consider the
case where the PCF of the initial generation $G_0$ and the noise $Z_n$
are affine expressions:
% Finally, in the general case where
\begin{equation}\label{e:PCFsuper}
	g_{G_0}-1=a_0+a_1f_{0,1}*\tilde f_{0,1}+\dots+a_kf_{0,k}*\tilde
	f_{0,k}
\end{equation}
and
\begin{equation}\label{e:PCFnoisesuper}
	g_{Z_n}-1
	=b_{n,0}+b_{n,1}f_{{Z_n},1}*\tilde f_{{Z_n},1}+\dots+b_{n,l}f_{{Z_n},l}*\tilde f_{{Z_n},l},
	\qquad n=1, 2,\ldots,
\end{equation}
for real constants $a_0,\ldots,a_k,b_{n,1},\ldots,b_{n,l}$ and PDFs
$f_{0,1},\ldots,f_{0,k},f_{{Z_n},1},\ldots,f_{{Z_n},l}$.  For
instance, the superposition of $k$ independent Poisson, weighted
permanental, or weigthed determinantal point processes has a PCF of
the form \eqref{e:PCFsuper} or \eqref{e:PCFnoisesuper}. 
Also Theorem~\ref{t:main} provides examples of PCFs of the form \eqref{e:PCFsuper} or \eqref{e:PCFnoisesuper}. 
Assuming \eqref{e:PCFsuper} and \eqref{e:PCFnoisesuper},
Theorem~\ref{t:main} is immediately generalised  by replacing
$af_0*\tilde{f}_0$ in \eqref{e:main1} by \eqref{e:PCFsuper}, $b_nf_{Z_n}*\tilde{f}_{Z_n}$  in \eqref{e:main4} by
\eqref{e:PCFnoisesuper}, and similarly for
$b_if_{Z_i}*\tilde{f}_{Z_i}$ in \eqref{e:main3}.

\section{Same reproduction system}\label{s:samesystem}

Throughout this section we assume the same reproduction system over
generations, that is, in items (a)--(d), $\beta_n=\beta$, $\nu_n=\nu$,
$f_n=f$, $p_n=p$, $q_n=q$ do not depend on $n$, $Z_1, Z_2,\ldots$ are IID
stationary point processes, so $\rho_{Z_n}=\rho_Z$ for $n=1,2,\ldots,$
and $\rho_{G_0}=\rho_{G_1}=\dots=\rho_G>0$.  Note that the noise
process $Z_n$ and the initial generation process $G_0$ need not be
Poisson processes and the offspring densities need not be Gaussian as
in Shimatani's paper \cite{shimatani:10}.
% \AC{As in \cite{felsenstein:75}, \cite{Kingman:77}, and
% \cite{shimatani:10} (see items (v)-(vii)) assume the same
% reproduction system over generations, that is in item (a) and (b),
% $\beta_n=\beta$, $\nu_n=\nu$, $f_n=f$, $p_n=p$ do not depend on
% $n\ge1$, and $\rho_{G_0}=\rho_{G_1}=\ldots=\rho_G>0$; Further, as in
% \cite{shimatani:10} (see item (viii)), in time (c) assume
% $Z_1, Z_2,\ldots$ are IID point processes, so $\rho_{Z_n}=\rho_Z$
% for $n\ge1$.  Having the same reproduction system over generations
% This implies by %the following. By
By \eqref{e:intensityXn}, we have either
%if $\rho_{Z} = 0$ then $\beta p + q = 1$; if $\rho_{Z} > 0$ then $\beta p + q < 1$.
%Thus, if $\rho_{Z} > 0 
	\begin{alignat}{3}
		\label{e:infinitecluster}
		\beta p + q&=1 &&\quad\mbox{and}\quad& \rho_Z&=0,
		\\
		\shortintertext{or}
		\label{e:finitecluster}
		\beta p + q&<1 &&\quad\mbox{and}\quad& \rho_Z&>0.
	\end{alignat}
In case of \eqref{e:finitecluster},
	\begin{equation}\label{e:int_finitecluster}
	\rho_G=\rho_Z/(1-\beta p - q).
	\end{equation}

\subsection{Limiting pair correlation function}\label{s:second main}

Under the assumptions above and in Theorem~\ref{t:main}, setting $0^0=1$, the PCF
simplifies after a straightforward calculation to
% assuming $g_{G_0}$ is of the form \eqref{e:PCFX0}, we have
%\com{\begin{equation}\label{e:PCFsamereproduction}
%	\begin{aligned}
%		g_{G_n}(u)-1 = 
%		{} &a f_0*\tilde f_0*\sum_{k=0}^{n}\binom{n}{k}q^{2(n-k)}(\beta p)^{2k}f^{*k}*\tilde f^{*k}(u)
%		\\
%		&+\frac{c}{\rho_G}\sum_{i=0}^{n-1}\sum_{k=0}^{i}\binom{i}{k}q^{2(i-k)}(\beta p)^{2(k+1)} f^{*(k+1)}*\tilde f^{*(k+1)}(u)
%		\\
%		&+\left(\frac{\rho_{Z}}{\rho_{G}}\right)^2bf_Z*\tilde{f}_Z*\sum_{i=0}^{n-1}\sum_{k=0}^{i}\binom{i}{k}q^{2(i-k)}(\beta p)^{2k} f^{*k}*\tilde f^{*k}(u),
%	\end{aligned}
%\end{equation}}{
\begin{equation}\label{e:PCFsamereproduction}
	\begin{aligned}
	g_{G_n}(u)-1 = 
		{} &a f_0*\tilde f_0*\sum_{k_1=0}^{n}\sum_{k_2=0}^{n}\binom{n}{k_1}\binom{n}{k_2}q^{2n-k_1-k_2}(\beta p)^{k_1+k_2}f^{*k_1}*\tilde f^{*k_2}(u)
		\\
		&+\left\{\frac{c(\beta p)^2f*\tilde{f}+\beta pq(f+\tilde f)}{\rho_G} + \left(\frac{\rho_{Z}}{\rho_{G}}\right)^2bf_Z*\tilde{f}_Z\right\} \\ &\hspace{4mm}*\sum_{i=0}^{n-1}\sum_{k_1=0}^{i}\sum_{k_2=0}^{i}\binom{i}{k_1}\binom{i}{k_2}q^{2i-k_1-k_2}(\beta p)^{k_1+k_2}f^{*k_1}*\tilde f^{*k_2}(u),
	\end{aligned}
\end{equation} %}
for $n=1,2,\ldots$, where
\begin{equation*}
	c=(\nu+\beta^2-\beta)/\beta^2\quad\mbox{if $\beta>0$},\qquad c=0\quad\mbox{if $\beta=0$},
\end{equation*}
$f^{*n}$ is the $n$-th convolution power of $f$ if $n>0$, and $f^{*0}*\tilde f^{*0}=\delta_0$.
% (in which case $G_0$ being a stationary Poisson process with
% intensity $\rho_G$ implies that $G_n$ is a stationary Poisson
% process with intensity $\rho_G$ for any integer $n\ge0$).  Here,
% $\lim_{n\rightarrow\infty}(\beta p)^{2n}a f_0*\tilde
% f_0*f^{*n}*\tilde f^{*n}(u)= 0$ (with uniform convergence in case of
% \eqref{e:finitecluster} if $f$ is bounded), and so
For instance, consider the case %$f_0\sim N_d(\tau^2)$,
$f\sim N_d(\sigma^2)$ and $f_Z\sim N_d(\kappa^2)$, and suppose $d\ge3$ in case of \eqref{e:infinitecluster}. 
%\com{Assume that either 
%	\begin{equation}\label{e:assump1}
%	(\beta p)^2+q^2<1,
%	\end{equation} 
%	or  
%	\begin{equation}\label{e:assump2}
%	(\beta p)^2+q^2=1\quad\mbox{and}\quad \beta pq>0, 
%	\end{equation} 
%	or  
%	\begin{equation*}\label{e:assump3}
%	(\beta p)^2+q^2=1\quad\mbox{and}\quad d\ge3.
%	\end{equation*} 
%	This assumption and the binomial formula}{
%This assumption and 
Then the binomial formula combined with either \eqref{e:infinitecluster} %and $d\geq 3$ 
or \eqref{e:finitecluster} %,} 
imply that the first double sum in \eqref{e:PCFsamereproduction} tends to 0 as $n\rightarrow\infty$, and hence
%\com{\begin{equation}\label{e:gGnormal}
%	\begin{aligned}
%	g_G(u)-1 :={}&\lim_{n\rightarrow\infty}g_{G_n}(u)-1
%	\\
%	= {} &\frac{c}{\rho_G}\sum_{i=0}^{\infty}\sum_{k=0}^{i}\binom{i}{k}
%	\frac{q^{2(i-k)}(\beta p)^{2(k+1)}}{\left\{4\pi(k+1)\sigma^2\right\}^{d/2}}
%	\exp\left\{-\frac{\|u\|^2}{4(k+1)\sigma^2}\right\}
%	\\
%	&+b\left(\frac{\rho_{Z}}{\rho_{G}}\right)^2\sum_{i=0}^{\infty}\sum_{k=0}^{i}\binom{i}{k}
%	\frac{q^{2(i-k)}(\beta p)^{2(k+1)}}{\left[4\pi \left\{(k+1)\sigma^2 + \kappa^2\right\}\right]^{d/2}}\\
%	&\hspace{3.65cm}\cdot\exp\left[-\frac{\|u\|^2}{4\left\{(k+1)\sigma^2 + \kappa^2\right\}}\right]
%	\end{aligned}
%	\end{equation}}{
\begin{align}\label{e:gGnormal}
		g_G(u)-1 :={}&\lim_{n\rightarrow\infty}g_{G_n}(u)-1 \nonumber\\
		= {} &\frac{c}{\rho_G}\sum_{i=0}^{\infty}\sum_{k_1=0}^{i}\sum_{k_2=0}^{i}\binom{i}{k_1}\binom{i}{k_2}
		\frac{q^{2i-k_1-k_2}(\beta p)^{2+k_1+k_2}}{\left\{2\pi(2+k_1+k_2)\sigma^2\right\}^{d/2}} \nonumber\\
		&\hspace{30mm} \cdot\exp\left\{-\frac{\|u\|^2}{2(2+k_1+k_2)\sigma^2}\right\} \nonumber\\
		&+\frac{2}{\rho_G}\sum_{i=0}^{\infty}\sum_{k_1=0}^{i}\sum_{k_2=0}^{i}\binom{i}{k_1}\binom{i}{k_2}
		\frac{q^{2i-k_1-k_2+1}(\beta p)^{1+k_1+k_2}}{\left\{2\pi(1+k_1+k_2)\sigma^2\right\}^{d/2}} \nonumber\\
		&\hspace{35mm} \cdot\exp\left\{-\frac{\|u\|^2}{2(1+k_1+k_2)\sigma^2}\right\} \nonumber\\
		&+\left(\frac{\rho_Z}{\rho_G}\right)^2\sum_{i=0}^{\infty}\sum_{k_1=0}^{i}\sum_{k_2=0}^{i}\binom{i}{k_1}\binom{i}{k_2}
		\frac{q^{2i-k_1-k_2}(\beta p)^{k_1+k_2}}{\left[2\pi\left\{(k_1+k_2)\sigma^2 + 2\kappa^2 \right\}\right]^{d/2}} \nonumber\\
		&\hspace{35mm} \cdot\exp\left[-\frac{\|u\|^2}{2\left\{(k_1+k_2)\sigma^2 + 2\kappa^2 \right\}}\right]
\end{align} %}
is finite.
Shimatani in \cite{shimatani:10} only showed that this is finite under the assumption 
 $d=2$, $b=q=0$, and
$c>0$. Then Shimatani noticed that $\beta p=1$ and $\rho_Z=0$ (which is \eqref{e:infinitecluster} for $q=0$) imply divergence of
$g_{G_n}$ as $n\to\infty$, whilst $\beta p<1$ and $\rho_Z>0$ (which is \eqref{e:finitecluster} for $q=0$) imply
convergence.  Further, in the case of convergence and when $\beta p\approx1$, he
discussed an approximation of $g_G(u)$ that depends on whether $\|u\|$
is close to 0 or not.

In general (i.e., without making the assumption of normal distributions and so on), if we assume $g_{G_n}-1$
has a finite limit and \eqref{e:finitecluster} %either \eqref{e:assump1} or \eqref{e:assump2} 
is satisfied, then
%\com{\begin{equation}\label{e:PCFequil}
%	\begin{aligned}
%	g_G(u)-1
%	={}&\frac{c}{\rho_G}\sum_{i=0}^{\infty}\sum_{k=0}^{i}\binom{i}{k}q^{2(i-k)}(\beta p)^{2(k+1)} f^{*(k+1)}*\tilde f^{*(k+1)}(u)
%	\\
%	&+\left(\frac{\rho_{Z}}{\rho_{G}}\right)^2bf_Z*\tilde{f}_Z*\sum_{i=0}^{\infty}\sum_{k=0}^{i}\binom{i}{k}q^{2(i-k)}(\beta p)^{2k} f^{*k}*\tilde f^{*k}(u)
%	\end{aligned}
%	\end{equation}}{
\begin{equation}\label{e:PCFequil}
\begin{aligned}
g_G(u)-1
={}&\left\{\frac{c(\beta p)^2f*\tilde{f}+\beta pq(f+\tilde f)}{\rho_G} + \left(\frac{\rho_{Z}}{\rho_{G}}\right)^2bf_Z*\tilde{f}_Z\right\} \\
&*\sum_{i=0}^{\infty}\sum_{k_1=0}^{i}\sum_{k_2=0}^{i}\binom{i}{k_1}\binom{i}{k_2}q^{2i-k_1-k_2}(\beta p)^{k_1+k_2} f^{*k_1}*\tilde f^{*k_2}(u)
\end{aligned}
\end{equation} %}
which does not depend on $a$ or $f_0$.  Here, as $\beta p +q\uparrow 1$,
$\rho_Z/\rho_G$ goes to 0, meaning that the less noise we consider,
the less it matters which type of PCF for the noise process $Z_n$ we choose. On the
other hand, as $\beta p+q\downarrow 0$, $g_G-1$ tends to
$bf_{Z}*\tilde{f}_{Z}$, which simply is the PCF of %the retained points of 
$Z_n$.

Considering the situation at the end of Section~\ref{s:PCFResults},
assume that $d=2$, $q=0$, $f\sim N_d(\sigma^2)$, and
$g_{Z_n}-1=bf_Z*\tilde{f}_Z$ (corresponding to \eqref{e:PCFX}) with
$f_Z\sim N_d(\kappa^2/8)$ and $b=0$, $b=-(\sqrt{\pi}\kappa)^2$, and
$b=2(\sqrt{\pi}\kappa)^2$ for the Poisson, determinantal, and weighted
permanental point process, respectively.  Then $g_G(u)$ is given by
\eqref{e:gGnormal}, where $d=2$ and $\kappa^2$ is replaced by
$\kappa^2/8$.  Also assume that $p=1$, $\sigma=0.1$, $\rho_G=100$, and
the number of points in a cluster is Poisson distributed (implying
$c=1$) with mean $\beta=0.8$, so $\rho_Z=20$.  Finally, assume
$\kappa=0.1$ in case of weighted permanental noise and
$\kappa = 1/\sqrt{\rho_Z\pi}$ in case of determinantal noise (the most
repulsive Gaussian determinantal point process). Shimatani in 
\cite{shimatani:10} discussed the case where $\beta p=0.99$
-- a plot (omitted here) shows that the limiting PCFs
corresponding to the three models of noise processes are then
effectively equal.  By lowering $\beta p$, the reproduction system is
diminished, and hence depending on the model type, a higher degree of
regularity or clustering is obtained.  This will also increase the
rate of convergence because the number of generations initialized by a
single point will be fewer. Note that in Figure~\ref{f:pcfslim} the
convergence is already rapid as $g_{G_8}$ and $g_{G_{16}}$ are
practically indistinguishable. Figure~\ref{f:pcfslim} further shows
that it is only for small or moderate inter-point distances that the three
limiting PCFs clearly differ.

\begin{figure}[ht]
	\centering
	\includegraphics[width=0.32\textwidth]{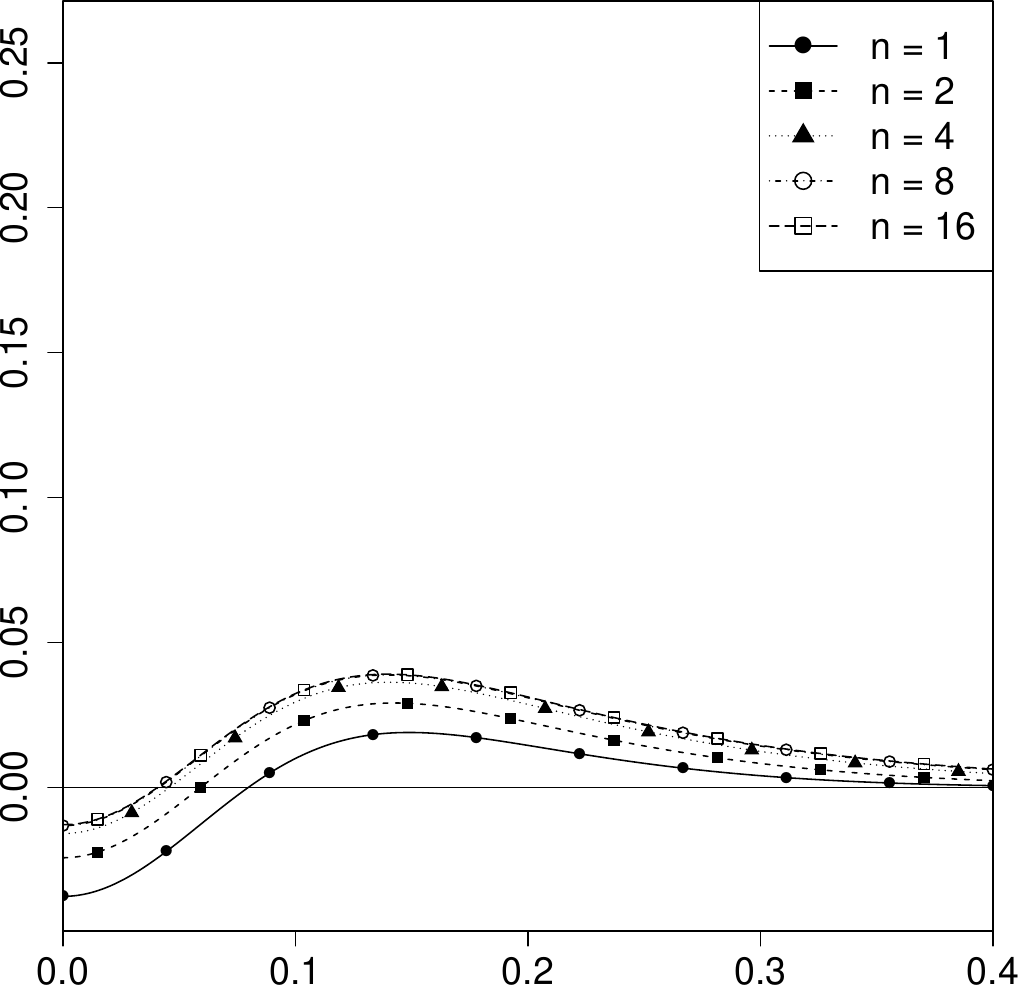}
	\hfill
	\includegraphics[width=0.32\textwidth]{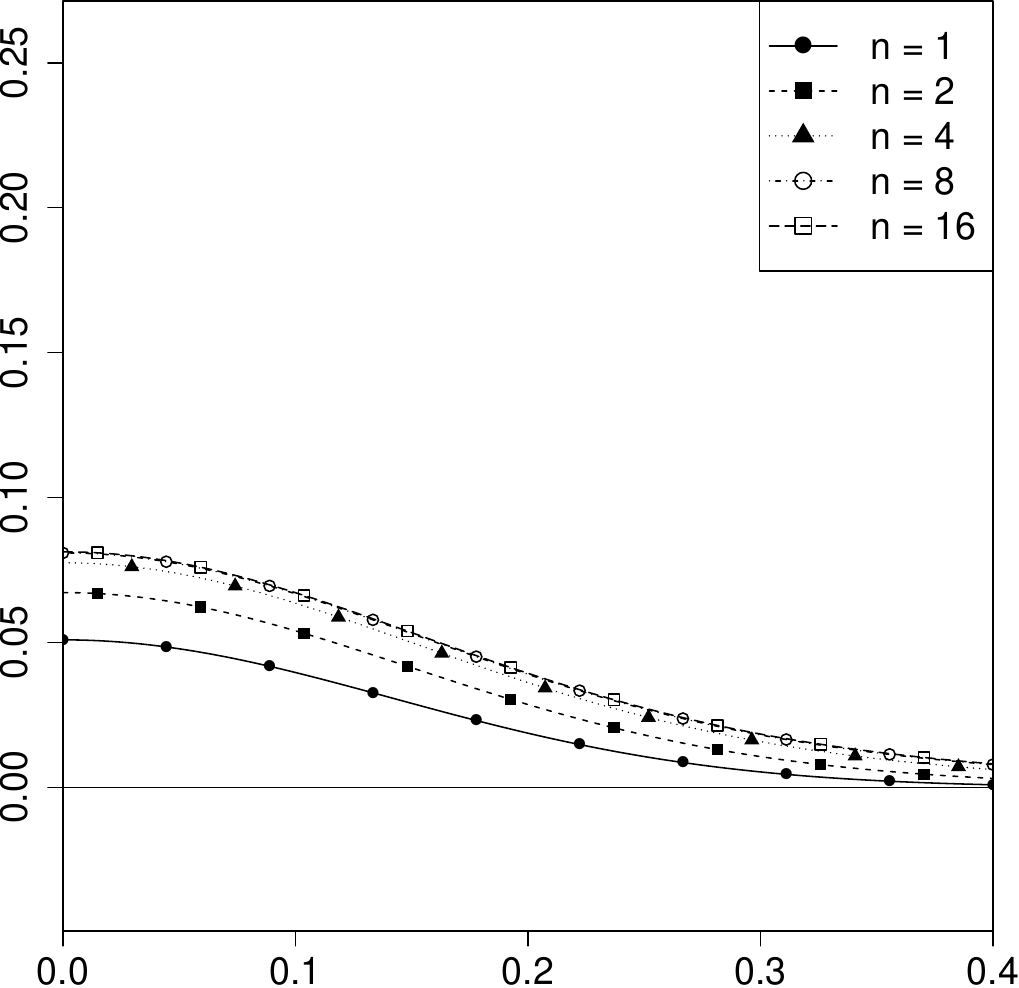}
	\hfill
	\includegraphics[width=0.32\textwidth]{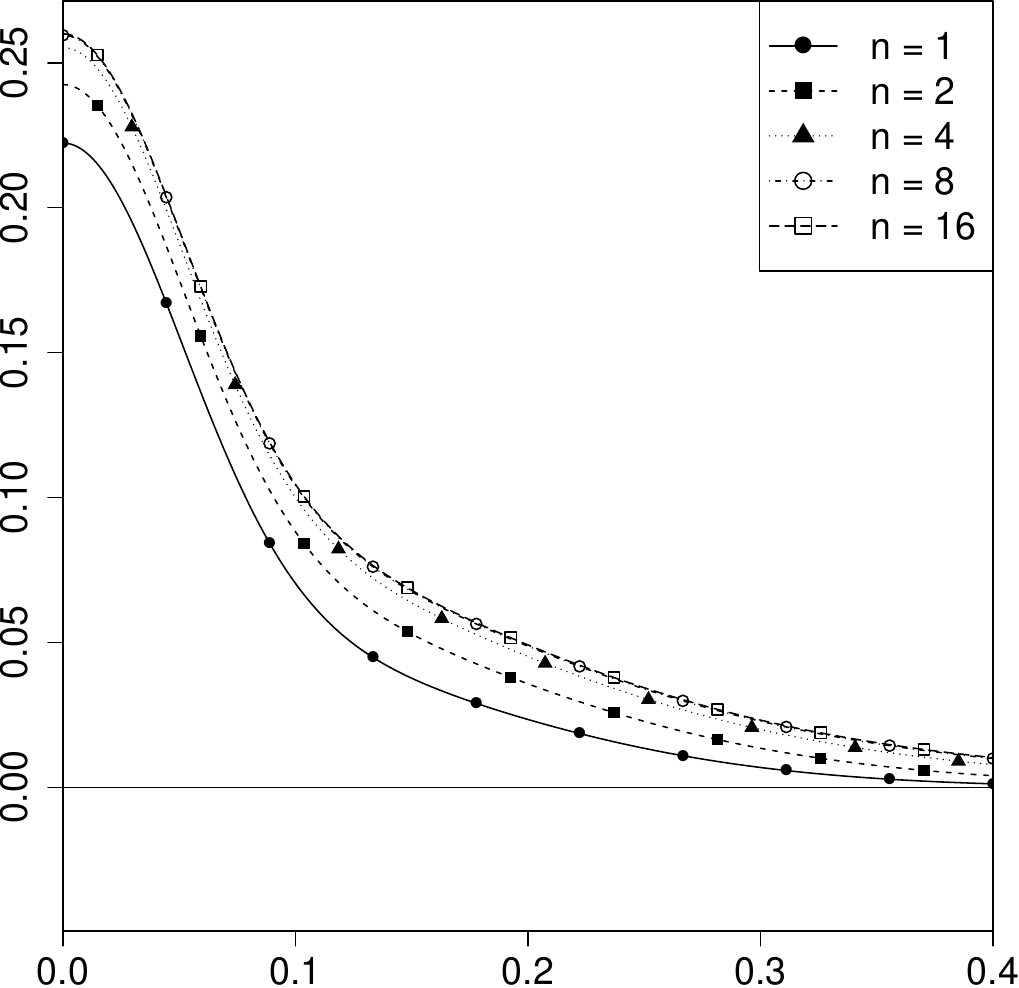}
	\caption{The reduced PCFs $g_{G_n}-1$ when the noise processes are either
		determinantal, Poisson or weighted permanental point processes
		(left to right), with parameters and Gaussian offspring PDF as
		specified in the text. The solid horizontal line is the
		 PCF - 1 for a Poisson process.}
	\label{f:pcfslim}
\end{figure}

\subsection{Second main result}
\label{s:secondmainresult}

Although Shimatani in \cite{shimatani:10} showed convergence of $g_{G_n}$ in the
special case considered above, he did
not clarify whether the Markov chain $G_0,G_1,\ldots$ converges in
distribution to a limit so that this limiting distribution (also
called the equilibrium, invariant, or stationary distribution) has a
PCF given by \eqref{e:PCFequil}.  In order to show that
$G_0,G_1,\ldots$ is indeed converging to a limiting distribution under
more general conditions, and to specify what this is, we construct in
accordance with items (a)--(d) a Markov chain
$\ldots,G_{-1}^{\mathrm{st}},G_0^{\mathrm{st}},G_1^{\mathrm{st}},\ldots$
with times given by all integers $n$ and so that this chain is
time-stationary (its distribution is invariant under discrete time
shifts), as follows.  First, we generate noise processes as in item
(d): Let $\ldots,Z_{-1},Z_0,Z_1,\ldots$ be independent stationary
Poisson processes on $\mathbb R^d$ with intensity $\rho_Z$. Second, for any integer $n$ and point $x\in Z_n$, we consider the family of
all generations initiated by the ancestor $x$, that is, the family
\begin{equation*}
	F_{n,x}=\bigcup_{m=1}^\infty
	W_{n,x}^{(m)}
\end{equation*}
where 
\[W_{n,x}^{(1)}=\begin{cases}W_{n,x}\cup \{x\}&\mbox{if }Q_{1,x}=1,\\ W_{n,x}&\mbox{if }Q_{1,x}=0,\end{cases}\]
is defined by the reproduction mechanism
of independent clustering, independent thinning, and independent retention given in items (a)--(c) (with $\beta_n=\beta$ and
$\nu_n=\nu$), $W_{n,x}^{(2)}$ is the offspring and retained points generated by
the points in $W_{n,x}^{(1)}$ (using the same reproduction mechanism
as before), and so on.
% the branching process
% \[F_{n,x}=\cup_{m=1}^\infty W_{n+m,x}\] which is the family of all
% generations initiated by the ancestor $x$, cf.\ items (a)-(b).
%% Note that all families are almost surely finite, with mean family
%% size
%% \[\mathrm E\left(\# F_{n,x}\right)=\sum_{m=1}^\infty(\beta
%%   p)^{m}=\beta p/(1-\beta p).\]
In other words, $W_{n,x}^{(m)}$ is the set of $(m+n)$-th generation
points with common $n$-th generation ancestor $x\in Z_n$.
Moreover, we assume that conditional on
$\ldots,Z_{-1},Z_0,Z_1,\ldots,$ the families $F_{n,x}$ for all
integers $n$ and $x\in Z_n$ are independent (and hence IID).
% Note that the expected family size is
% \begin{align*}
%   {\rm EFS} :=\mathrm E\left(\#F_{n,x}\right) =
%   \sum_{m=1}^{\infty}\int(\beta p)^mf^{*m}(y-x){\rm dy}
%%   = \sum_{m=1}^{\infty}(\beta p)^m
%   =\frac{\beta p}{1 - \beta p}
% \end{align*}
% when \eqref{e:finitecluster} is satisfied.  {\sf JM: Ovenstaaende
% skal ud, hvis vi aldrig faar brugt EFS til noget... ligesom
% "$\gamma$-stoffet" skal checkes om det bruges til noget.}
Finally, for all integers $n$, we let
\begin{equation}\label{e:statext}
	G_n^{\mathrm{st}}=W_n^{\mathrm{st}}\cup Z_n\qquad\mbox{with } W_n^{\mathrm{st}}=\bigcup_{m=1}^\infty\bigcup_{x\in Z_{n-m}}W_{n-m,x}^{(m)}.
\end{equation}

It will be evident from the next theorem that any
$G_n^{\mathrm{st}}$ has intensity $\rho_G$ given by
\eqref{e:int_finitecluster} and PCF $g_G$ given by \eqref{e:PCFequil} (provided $g_G(u-v)$ is a locally integrable function of
$(u, v)\in\mathbb{R}^d\times\mathbb{R}^d$); a formal proof is given in Appendix~\ref{a:2}.
%For completeness, we show in Appendix~\ref{a:2} that any
%$G_n^{\mathrm{st}}$ has intensity $\rho_G$ given by
%\eqref{e:int_finitecluster} and PCF $g_G$ given by \eqref{e:PCFequil},
%although this should be evident from Theorem~\ref{t:main2} below. 
The
proof of Theorem~\ref{t:main2} is based on a coupling construction
between $G_1,G_2,\ldots$ and
$G_1^{\mathrm{st}},G_2^{\mathrm{st}},\ldots$ together with the
following result.

\begin{lemma}\label{l:3}
	Suppose $\beta_n=\beta$, $\nu_n=\nu$, $f_n=f$, $p_n=p$, $q_n=q$, and
	$\rho_{Z_n}=\rho_{Z}$ do not depend on $n\ge1$, where $\beta p + q<1$
	and $\rho_{Z}>0$. Let $K\subset\mathbb R^d$ be a compact set and let
	\begin{equation}\label{e:T0Kst}
		T_{0,K}^{\mathrm{st}}=\sup\bigl\{m\in\{1,2,\ldots\}:W_{0,x}^{(m)}\cap
		K\not=\emptyset \text{ for some }x\in G_0^{\mathrm{st}}\bigr\}
	\end{equation}
	be the last time a point in $K$ is a member of a family initiated by
	some point in the $0$-th generation $G_0^{\mathrm{st}}$.  Then
		\begin{equation*}
			\mathrm E\left(T_{0,K}^{\mathrm{st}}\right)\le |K|\rho_G\frac{\beta p + q}{1 - \beta p - q}
		\end{equation*}
	is finite, and so $T_{0,K}^{\textup{st}}<\infty$ almost surely.
\end{lemma}

\begin{proof}
	Let $K\subset\mathbb R^d$ be compact and define
	\begin{equation*}
		N=\sum_{x\in G_0^{\mathrm{st}}}\#(F_{0,x}\cap K).
	\end{equation*}
	By the law of total expectation, conditioning on $G_0^{\text{st}}$ and using Campbell's theorem, we obtain
	\begin{align}
	\mathrm E(N)
	&=\rho_G\int\sum_{m=1}^\infty\int_K\big(\beta pf+q\delta_0\big)^{*m}(y-x)\,\mathrm dy\,\mathrm dx\nonumber\\
&=	
	\rho_G\int \sum_{m=1}^\infty \int_K\sum_{k=0}^{m}\binom{m}{k}q^{m-k}(\beta p)^k f^{*k}(y-x)\,\mathrm dy  \,\mathrm dx \nonumber\\
	&=|K|\rho_G\frac{\beta p + q}{1 - \beta p - q}\label{e:EN}
	\end{align}
	using Fubini's theorem in the last identity.  Further, the families
	initiated by the points in $G_0^{\mathrm{st}}$ are almost surely
	pairwise disjoint, so $N$ is almost surely the number of points in
	$K$ belonging to some family initiated by a point
	$x\in G_0^{\mathrm{st}}$. Consequently,
	$\mathrm P(T_{0,K}^{\mathrm{st}}\le N)=1$, whereby the lemma
	follows.
\end{proof}

We are now ready to state our second main result.
\begin{theorem}\label{t:main2}
	Suppose $\ldots,Z_{-1},Z_0,Z_{1}\ldots$ are IID stationary point
	processes and $\beta_n=\beta$, $\nu_n=\nu$, $f_n=f$, $p_n=p$, $q_n = q$, and
	$\rho_{Z_n}=\rho_{Z}$ do not depend on $n\ge1$, where $\beta p + q<1$
	and $\rho_{Z}>0$.  Then
	$\ldots,G_{-1}^{\mathrm{st}},G_0^{\mathrm{st}},G_1^{\mathrm{st}},\ldots$
	is a time-stationary Markov chain constructed in accordance to items
	(a)--(d). Let $\Pi$ be the distribution of any $G_n^{\mathrm{st}}$
	and let $\mathcal N$ be the space of all locally finite subsets of
	$\mathbb R^d$. Then there exists a (measurable) subset
	$\Omega\subseteq\mathcal N$ so that $\Pi(\Omega)=1$ and for any
	compact set $K\subset\mathbb R^d$ and all $\omega\in\Omega$,
	conditional on $G_0=\omega$, there is a coupling between
	$G_1,G_2,\ldots$ and
	$\ldots,G_{-1}^{\mathrm{st}},G_0^{\mathrm{st}},G_1^{\mathrm{st}},\ldots$,
	and there exists a random time $T_K(\omega)\in\{0,1,\ldots\}$ so
	that $G_n\cap K=G_n^{\mathrm{st}}\cap K$ for all integers
	$n> T_K(\omega)$.  In particular, for any $\omega\in\Omega$ and
	conditional on $G_0=\omega$,
	% the chain $G_0,G_1,\ldots$ is ergodic when the initial
	% distribution has support included in $\Omega$, that is, for any
	% $\omega\in\Omega$ and
	% conditional on $G_0=\omega$, as $n\rightarrow\infty$ then
	$G_n$ converges in distribution to $\Pi$ as $n\rightarrow\infty$,
	and so $\Pi$ is the unique invariant distribution of the chain
	$G_0,G_1,\ldots$.
	% \begin{equation}\label{e:equiPCF}
	% g_{G_0}(u)-1=g_{\infty}-1=...
	%% \frac{ac}{\rho_G}\sum_{n=1}^\infty \left\{\frac{(\beta
	%% p)^2}{a}\right\}^n f^{*n}*\tilde f^{*n}=
	%% \frac{c+\rho_G}{\rho_G} (\beta p)^2 \sum_{n=1}^\infty q_n
	%% f^{*n}*\tilde f^{*n}
	% \end{equation}
\end{theorem}

\begin{proof}
	Obviously,
	$\ldots,G_{-1}^{\mathrm{st}},G_0^{\mathrm{st}},G_1^{\mathrm{st}},\ldots$
	is a time-stationary Markov chain con\-struc\-ted in accordance to
	items (a)--(d). To verify the remaining part of the theorem, we may
	assume that $G_0$ and $G_0^{\mathrm{st}}$ are independent. Then,
	conditional on $G_0$, we have a coupling between $G_1,G_2,\ldots$
	and
	$\ldots,G_{-1}^{\mathrm{st}},G_0^{\mathrm{st}},G_1^{\mathrm{st}},\ldots$
	because $G_1^{\mathrm{st}},G_2^{\mathrm{st}},\ldots$ and
	$G_1,G_2,\ldots$ are generated by the same noise processes
	$Z_1,Z_2,\ldots$, the same offspring processes $Y_{n,x}$ for all
	times $n=1,2,\ldots$ and all ancestors
	$x\in G_{n-1}\cap G_{n-1}^{\mathrm{st}}$, the same Bernoulli
	variables $B_{n,y}$ for all times $n=1,2,\ldots$ and all offspring
	$y\in Y_{n,x}$ with ancestor
	$x\in G_{n-1}\cap G_{n-1}^{\mathrm{st}}$, and the same Bernoulli
		variables $Q_{n,x}$ for all times $n=1,2,\ldots$ and all retained points
		$x\in G_{n-1}\cap G_{n-1}^{\mathrm{st}}$. 
	Let $K\subset\mathbb R^d$
	be compact. In accordance with \eqref{e:T0Kst}, for $\omega\in \mathcal{N}$, let
	\begin{equation*}
		T_{K}(\omega)=\sup\bigl\{m\in\{1,2,\ldots\}:W_{0,x}^{(m)}\cap
		K\not=\emptyset\text{ for some }x\in\omega\bigr\}
	\end{equation*}
	be the last time a point in $K$ is a member of a family initiated by
	some point in $\omega$, and let
	$\Omega=\{\omega\in\mathcal N:T_{K}(\omega)<\infty\}$. By
	Lemma~\ref{l:3} and the coupling construction, $\Pi(\Omega)=1$ and
	$G_n\cap K = G_n^{\mathrm{st}}\cap K$ whenever $n > T_K(\omega)$, so
	for any $\omega\in\Omega$,
	\begin{align*}
		\lim_{n\rightarrow\infty}\mathrm P\left(G_n\cap
		K=\emptyset|G_0=\omega\right)= \lim_{n\rightarrow\infty}\mathrm
		P\left(G_n^{\mathrm{st}}\cap K=\emptyset, n>T_K(\omega)\right)
	\end{align*}
	because $G_0$ is independent of
	$(G_0^{\mathrm{st}},T_K(\omega))$. Since the sequence of events
	$\{\omega:1>T_K(\omega)\}\subseteq
	\{\omega:2>T_K(\omega)\}\subseteq\ldots$ increases to $\Omega$, we
	obtain
	\begin{equation*}
		\lim_{n\rightarrow\infty}\mathrm P\left(G_n\cap K=\emptyset|G_0=\omega\right)
		=\lim_{n\rightarrow\infty}\mathrm P\left(G_n^{\mathrm{st}}\cap K=\emptyset\right)
		=\mathrm P\left(G_0^{\mathrm{st}}\cap K=\emptyset\right).
	\end{equation*}
	Thus, recalling that the distribution of a random closed set
	$X\subseteq\mathbb R^d$ (e.g.\ a locally finite point process)
	is uniquely characterized by the void probabilities
	$\mathrm P(X\cap K=\emptyset)$ for all compact sets
	$K\subset\mathbb R^d$, we have verified that conditional on
	$G_0=\omega$, the chain $G_1,G_2\ldots$ converges in distribution
	towards $\Pi$. In turn, this implies uniqueness of the invariant
	distribution $\Pi$.
\end{proof}

In Theorem~\ref{t:main2}, under mild conditions, we can take
$\Omega=\mathcal N$. For instance, this is easily seen to be the case
if there exists $\varepsilon>0$ so that $f(x)>0$ whenever
$\|x\|\le\varepsilon$. In the special case $c=0$, $\Pi$ is just a
stationary Poisson process, and so $\Omega=\mathcal N$. Moreover, the
integral
\begin{align*}
	\gamma := \int(g_G-1)
\end{align*}
is a rough measure of the amount of positive/negative association
between the points in $G_n^{\mathrm{st}}$. Note that comparing
$\gamma$ with the corresponding measure for another stationary point
process makes only sense if the processes have equal intensities, see
\cite{LMR15}.  Under the assumptions in both Theorem~\ref{t:main} and
\ref{t:main2}, by \eqref{e:PCFequil},
%\com{\begin{align*}
%	\gamma &=\frac{c(\beta p)^2}{\rho_G\left\{1-(\beta p)^2-q^2\right\}}
%	+\frac{b\rho_Z^2}{\rho_G^2\left\{1-(\beta p)^2-q^2\right\}} \\
%	&=\frac{1-\beta p-q}{1-(\beta p)^2-q^2}\left\{\frac{c(\beta p)^2}{\rho_Z}+b(1-\beta p-q)\right\}
%	\end{align*}}{
\begin{align*}
\gamma &=\frac{c(\beta p)^2+2\beta pq}{\rho_G\left\{1-(\beta p+q)^2\right\}}
+\frac{b\rho_Z^2}{\rho_G^2\left\{1-(\beta p+q)^2\right\}} \\
&=\frac{1}{1+\beta p +q}    %{1-\beta p-q}{1-(\beta p+q)^2}
\left\{\frac{c(\beta p)^2+2\beta pq}{\rho_Z}+b(1-\beta p-q)\right\}
\end{align*} %}
which does not depend on $f$ or $f_Z$. Furthermore, $\gamma$ may take
any positive value and some negative values depending on how we choose
the values of the parameters. %$(\beta,\nu,p,b,\rho_Z)$. 
This means we may have an
equilibrium distribution exhibiting any degree of clustering or some
degree of regularity. In fact, $\gamma$ can only be negative when $b$
is negative, e.g\ when $Z_n$ is a determinantal point process. In this
case $b$ has a lower bound, $b_{\textup{min}}$, that ensures the existence
of the determinantal point process \cite{LMR15} and consequently,
$\gamma\ge b_{\textup{min}}$. The case $\gamma = b_{\textup{min}}$ happens
exactly when $\beta p + q= 0$ (i.e., when offspring are never produced or no points are retained after the thinning procedures in items (b) and (c)) and thus $G_n = Z_n$ is a determinantal
point process.

For approximate simulation of $G_0^{\textup{st}}$ under each of the three
models of the noise processes, we use the algorithm described in
Appendix~\ref{a:3}.  Simulation was initially done with parameters and
set-up corresponding to that of Figure~\ref{f:pcfslim}.  However, the
resulting point patterns were not distinguishable from a stationary
Poisson process when comparing empirical estimates of the PCF,
$L$-function, or $J$-function of the simulations to 95\% global rank
envelopes under each model (for definition of $L$- and $J$-functions,
see e.g.\ \cite{MW2004}, and for the envelopes, see
\cite{Myllymaki:16}).  Therefore, in order to better distinguish the
three models, we consider two cases as follows.
\begin{description}[style=nextline]
	\item[Case 1:] This case is based on minimizing $\gamma$ under
	determinantal noise and on maximizing $\gamma$ under weighted
	permanental noise. Let $d=2$, $f\sim N_d(\sigma^2)$, with
	$\sigma = 0.1$, $f_Z\sim N_d(\kappa^2/8)$, $\rho_G = 100$, $p = 1$,
	$\beta = 0.3$, $q=0$, and consequently $\rho_Z = 70$.
	\begin{itemize}
		\item In case of determinantal noise: Let
		$\kappa = 1/\sqrt{\rho_Z\pi}$ (the most repulsive Gaussian
		determinantal point process) and the number of points in a cluster
		be Bernoulli distributed with parameter $\beta$, implying $c=0$
		(each point has at most one offspring). Then
		$\gamma\approx -5.38\times10^{-3}$.
		\item In case of Poisson noise: Let the number of points in a
		cluster be Poisson distributed with intensity $\beta$, implying
		$c=1$. Then $\gamma\approx 9.89\times10^{-4}$.
		\item In case of weighted permanental noise: Let $\kappa = 1$ and
		the number of points in a cluster be negative binomially
		distributed with probability of success equal to 0.12 and
		dispersion parameter equal to 0.11, implying $c=10$.  Then
		$\gamma\approx 3.39$.
	\end{itemize}
	\item[Case 2:] This case is such that the clusters are more separated.
	Let $d=2$, $f\sim N_d(\sigma^2)$, with $\sigma = 0.01$,
	$f_Z\sim N_d(\kappa^2/8)$, $\rho_G = 100$, $p = 1$, $\beta = 0.95$, $q=0$,
	and consequently $\rho_Z = 5$. Also, let the number of points in a
	cluster be negative binomially distributed with probability of
	success equal to 0.208 and dispersion parameter equal to 0.25,
	implying $c=5$.
	\begin{itemize}
		\item In case of determinantal noise: Let
		$\kappa = 1/\sqrt{\rho_Z\pi}$. Then $\gamma\approx 0.463$.
		\item In case of Poisson noise: $\gamma\approx 0.463$.
		\item In case of weighted permanental noise: Let $\kappa = 1$. Then
		$\gamma\approx 0.624$.
	\end{itemize}
\end{description}

Figure~\ref{f:limsim} shows simulations of $G_0^{\textup{st}}$ under
each of the three models of the noise processes (left to right) in
Case 1 and 2 (top and bottom).  Based on these simulations,
Figure~\ref{f:limsim_summaries} shows empirical estimates of
functional summary statistics based on the simulated point patterns
from Figure~\ref{f:limsim} along with 95\% global rank envelopes based
on 2499 simulations (as recommended in \cite{Myllymaki:16}) of a
stationary Poisson process with the same intensity as used in
Figure~\ref{f:limsim}.  The first simulated point pattern of Case 1
looks slightly less clustered than the second, whilst the last looks
more clustered.  This is in accordance with the values of $\gamma$ and
the corresponding functional summary statistics in
Figure~\ref{f:limsim_summaries}.  Additionally,
Figure~\ref{f:limsim_summaries} reveals that the case of Poisson noise
is not distinguishable from the stationary Poisson process, while the
case of weighted permanental noise is more clustered. The case of
determinantal noise is not distinguishable from the stationary Poisson
process by the PCF or $L$-function, but is shown to be more regular by
the $J$-function.  In Case 2, the clusters of the point pattern
simulated under determinantal noise looks more separated than the
clusters of the point pattern simulated under Poisson noise. The
clusters of the point pattern simulated under weighted permanental
noise are clustered to such a degree that it gives the illusion of few
highly separated clusters.  All three models of Case 2 are as expected
significantly different from the stationary Poisson process.

\begin{figure}[ht]
	\centering
	\begin{minipage}{.33\linewidth}
		\includegraphics[width=\linewidth]{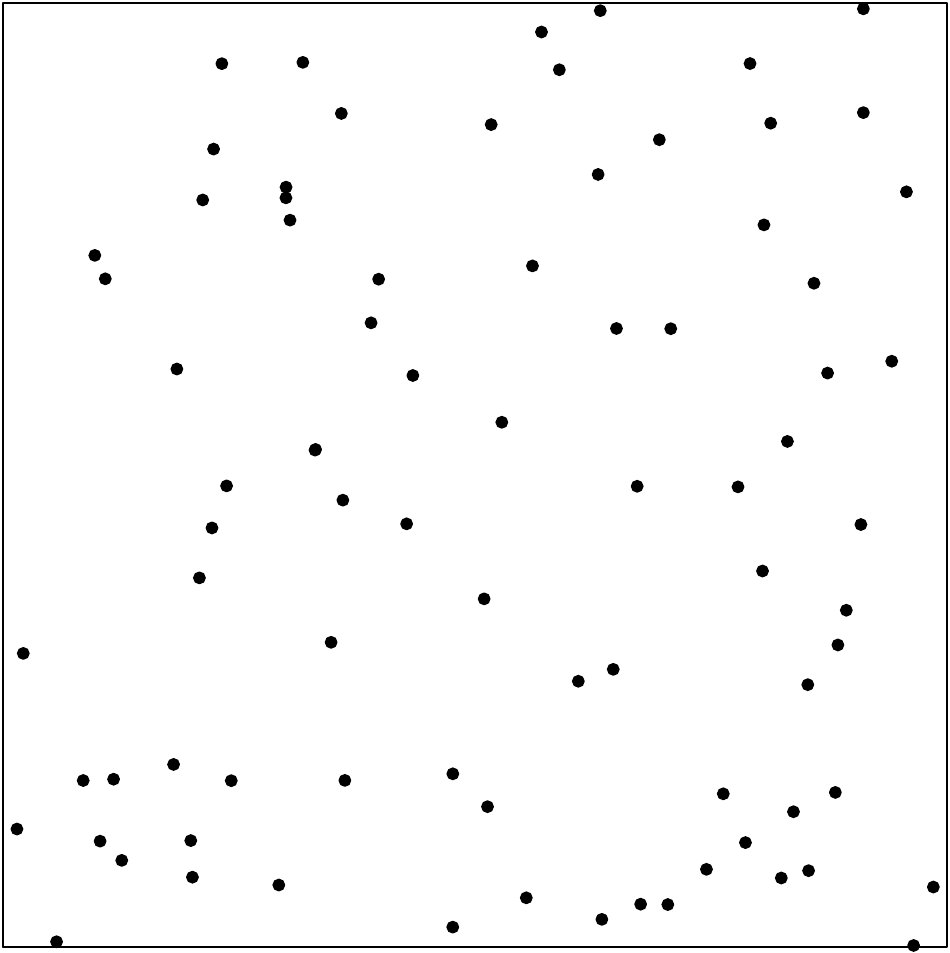}
	\end{minipage}%
	\begin{minipage}{.33\linewidth}
		\includegraphics[width=\linewidth]{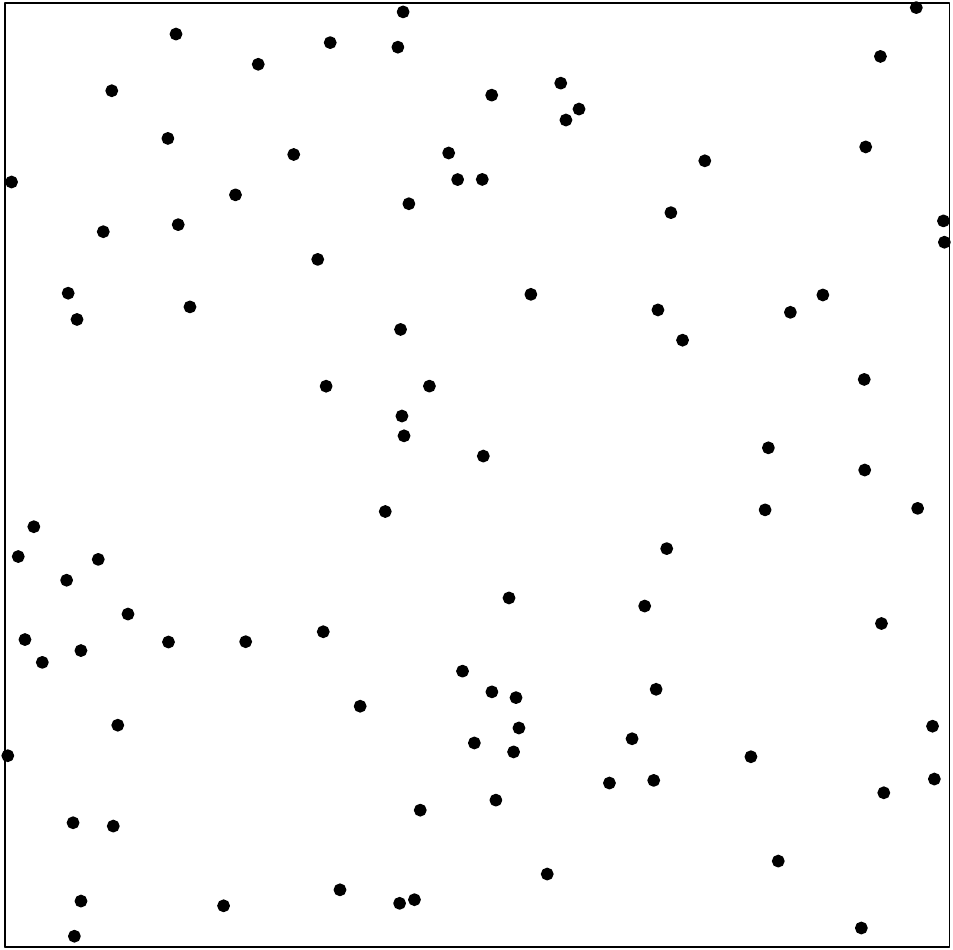}
	\end{minipage}%
	\begin{minipage}{.33\linewidth}
		\includegraphics[width=\linewidth]{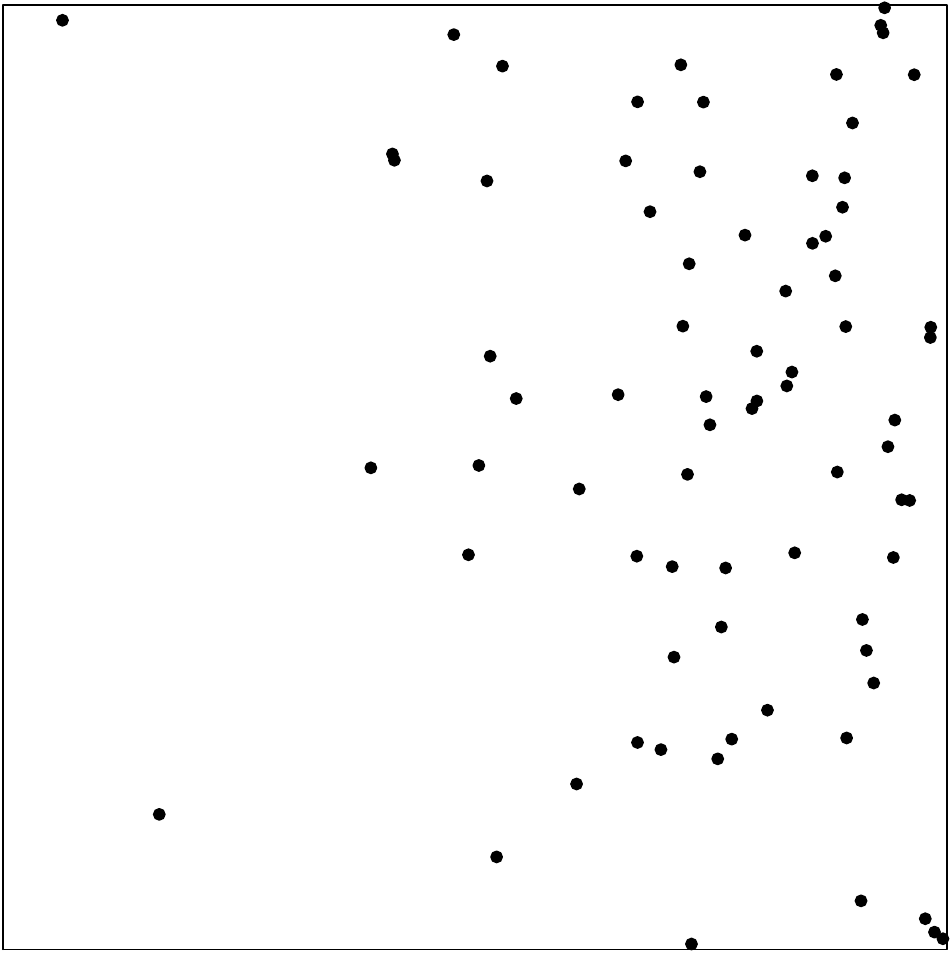}
	\end{minipage}%
	\\
	\begin{minipage}{.33\linewidth}
		\includegraphics[width=\linewidth]{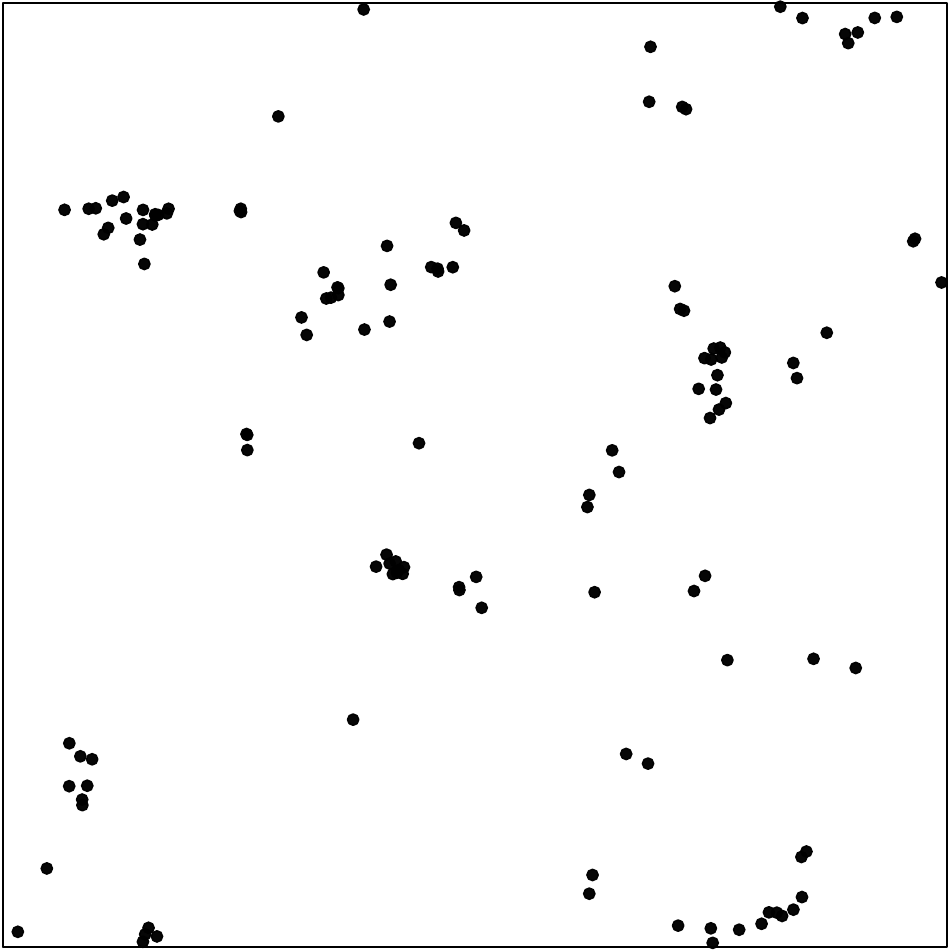}
	\end{minipage}%
	\begin{minipage}{.33\linewidth}
		\includegraphics[width=\linewidth]{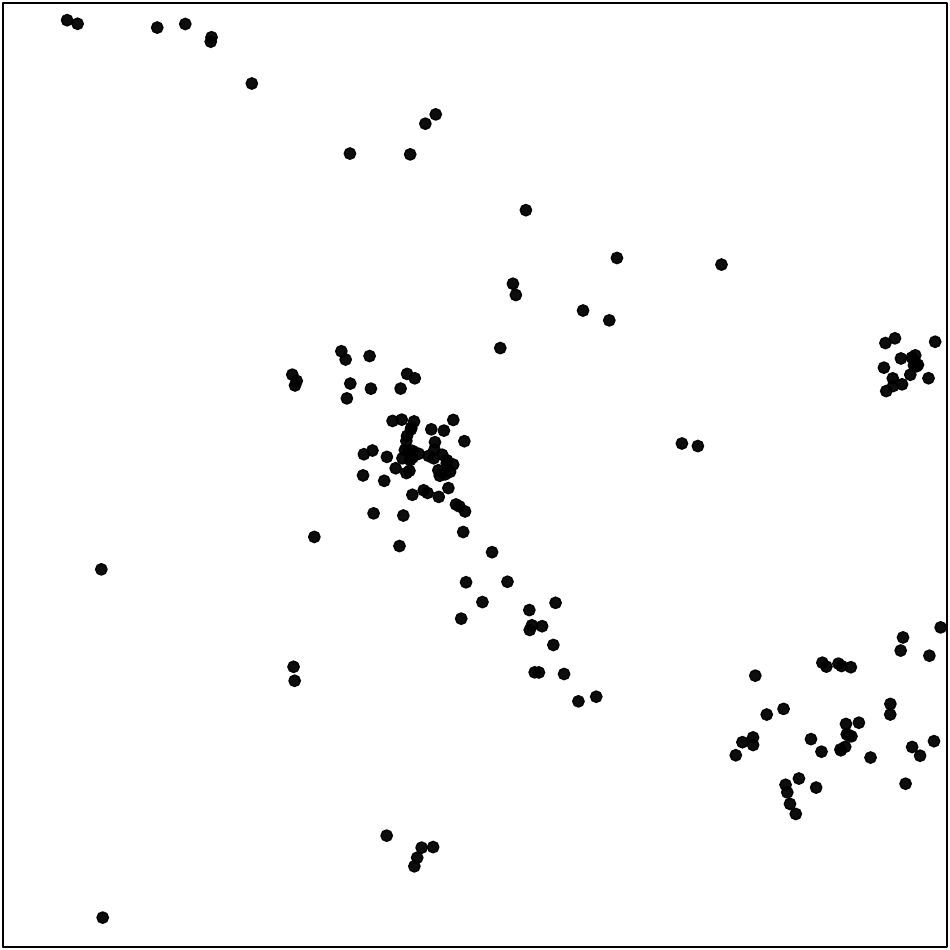}
	\end{minipage}%
	\begin{minipage}{.33\linewidth}
		\includegraphics[width=\linewidth]{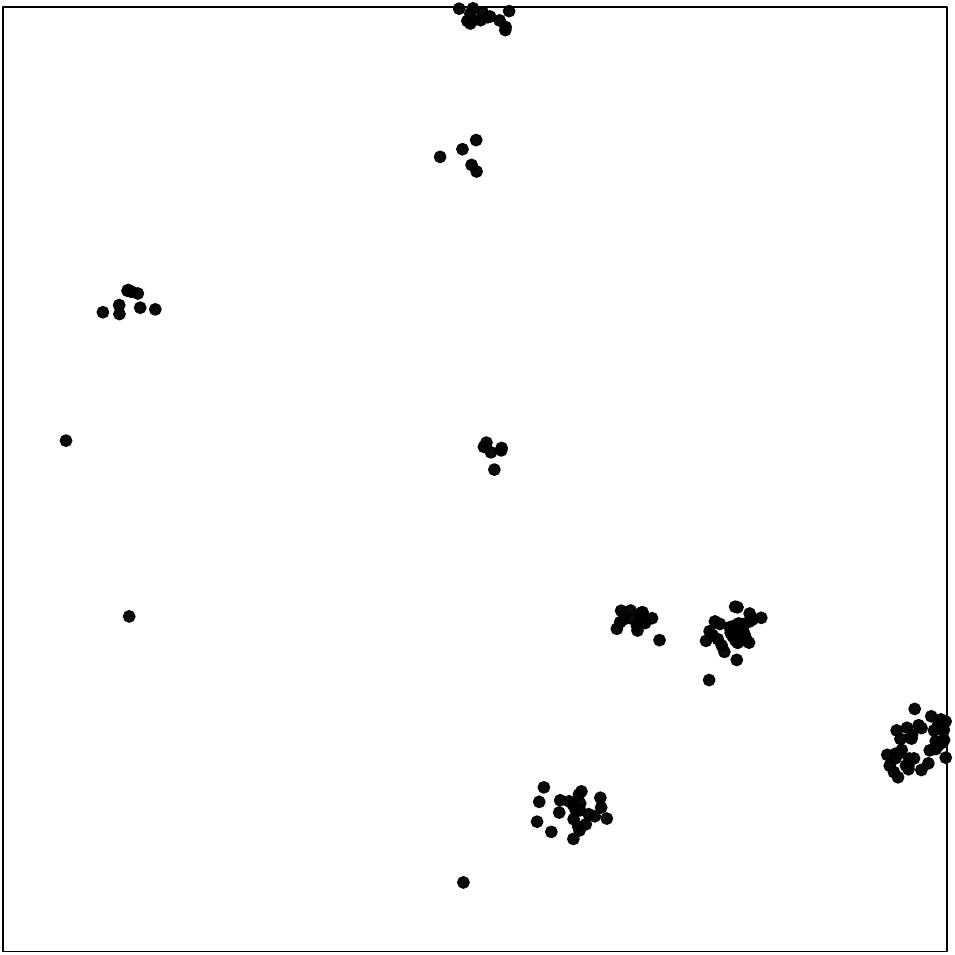}
	\end{minipage}%
	\caption{Simulations of $G_0^{\textup{st}}$ restricted to a unit
		square when the noise processes are either determinantal (left
		panel), Poisson (middle panel), or weighted permanental (right
		panel) point processes, with parameters as specified in the
		text. The rows corresponds to Case~1 and 2, respectively.}
	\label{f:limsim}
\end{figure}

\begin{figure}[ht]
	\centering
	\begin{minipage}{.33\linewidth}
		\includegraphics[width=\linewidth]{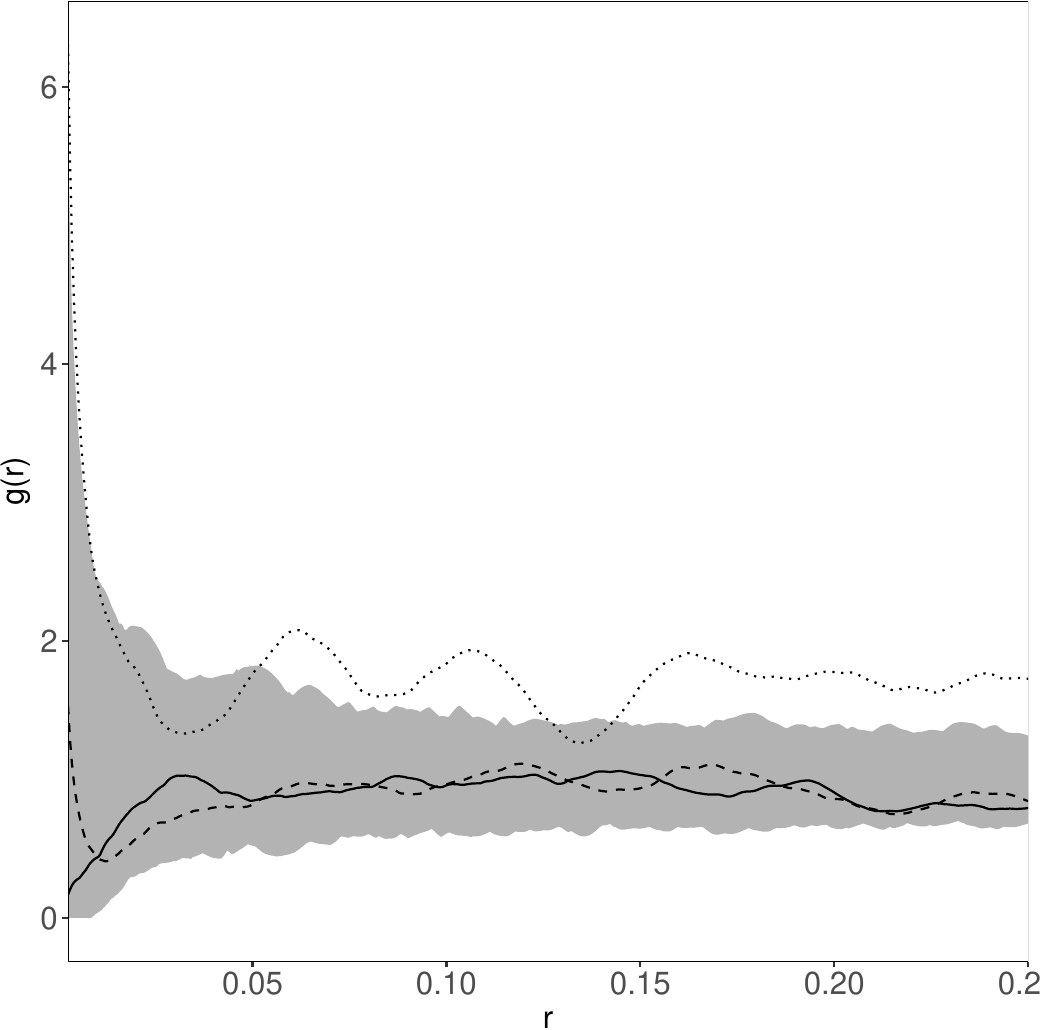}
	\end{minipage}%
	\begin{minipage}{.33\linewidth}
		\includegraphics[width=\linewidth]{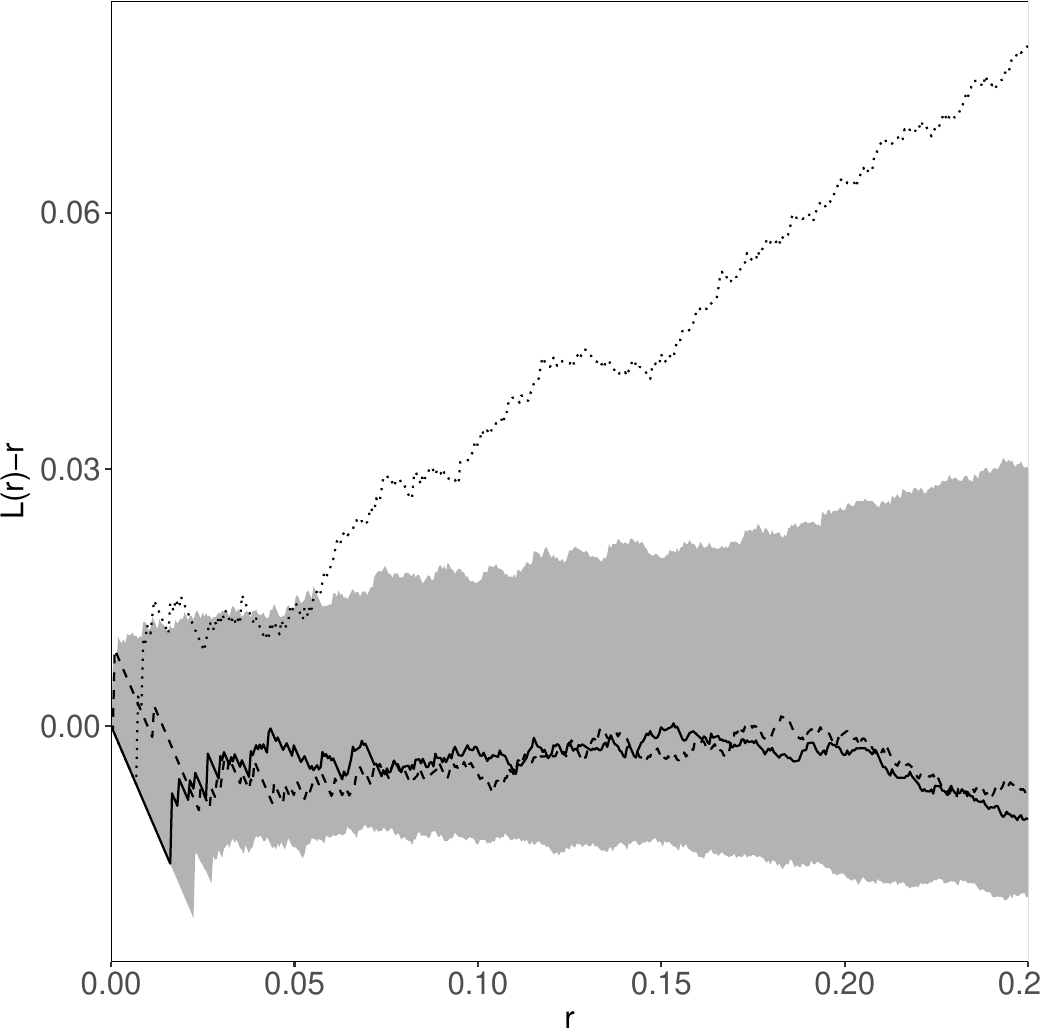}
	\end{minipage}%
	\begin{minipage}{.33\linewidth}
		\includegraphics[width=\linewidth]{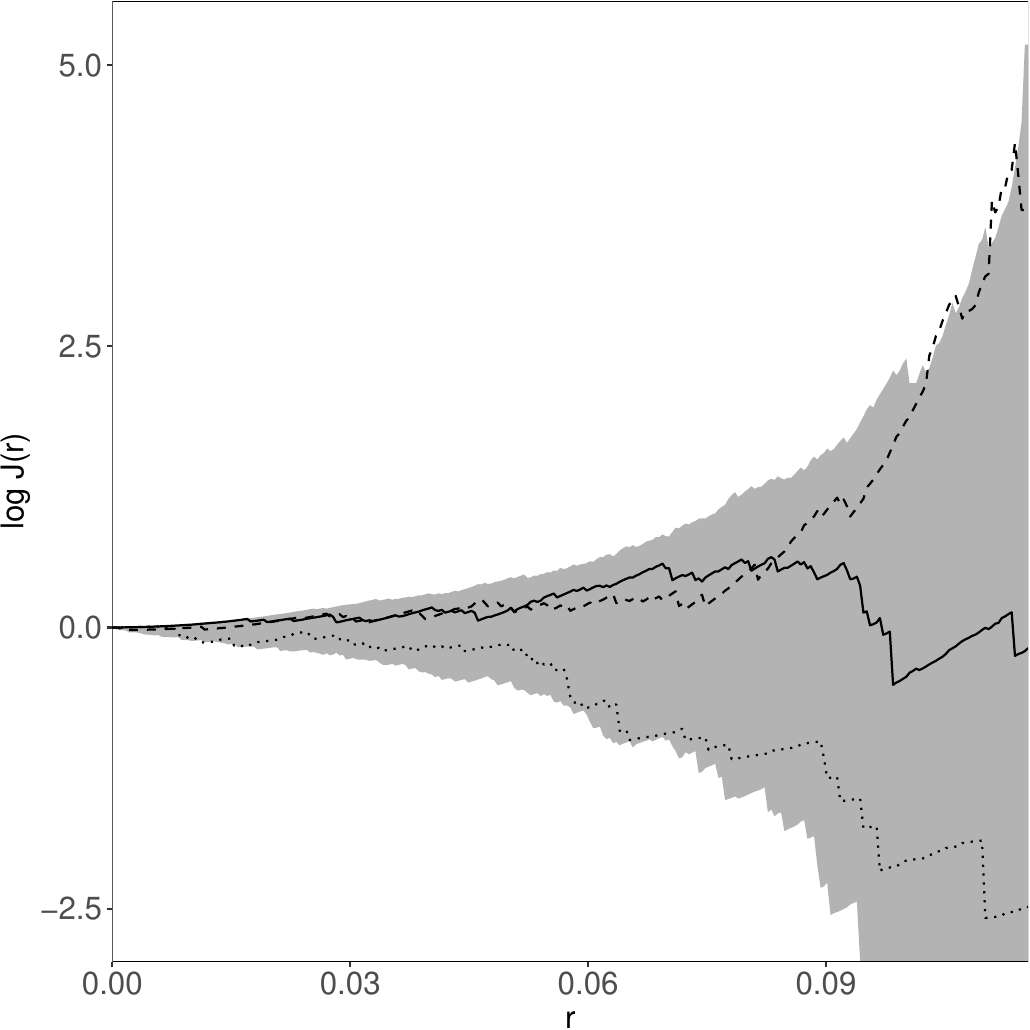}
	\end{minipage}%
	\\
	\begin{minipage}{.33\linewidth}
		\includegraphics[width=\linewidth]{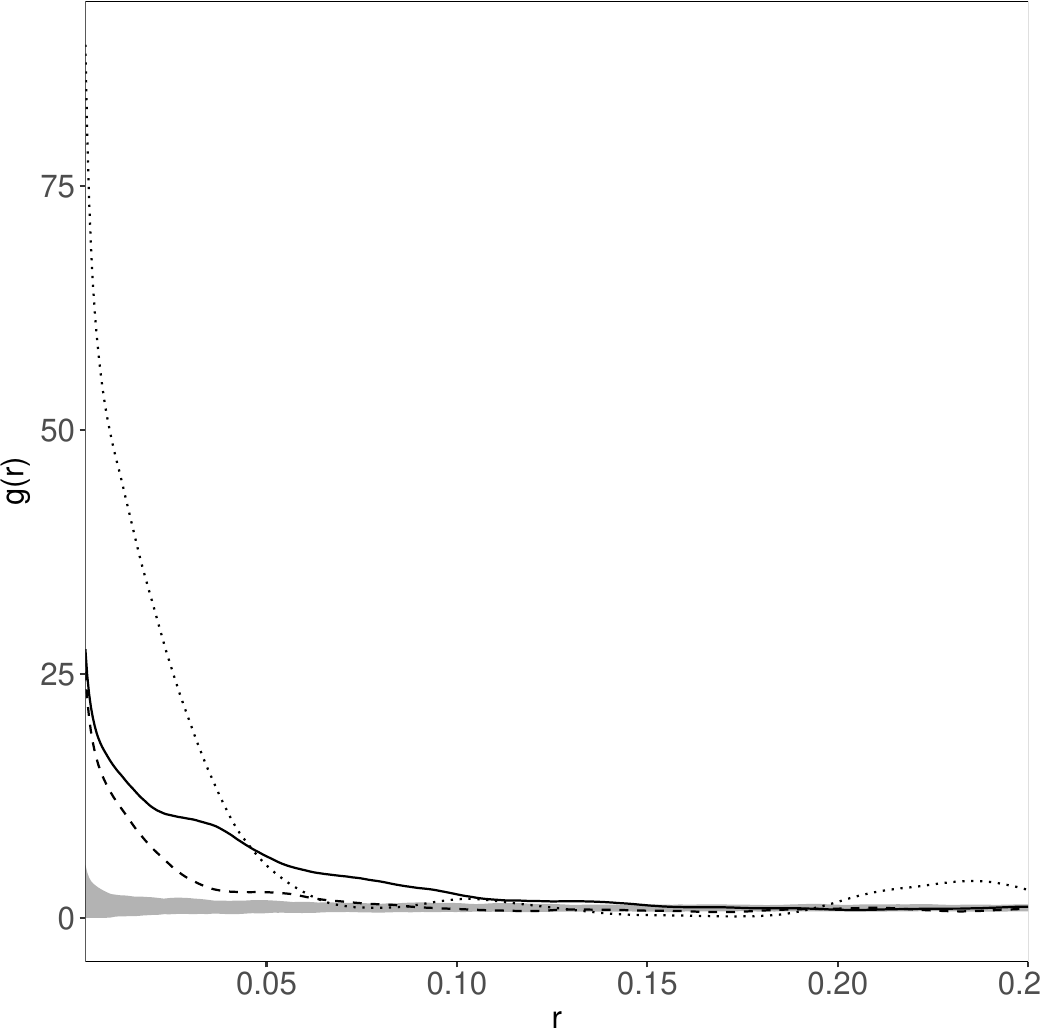}
	\end{minipage}%
	\begin{minipage}{.33\linewidth}
		\includegraphics[width=\linewidth]{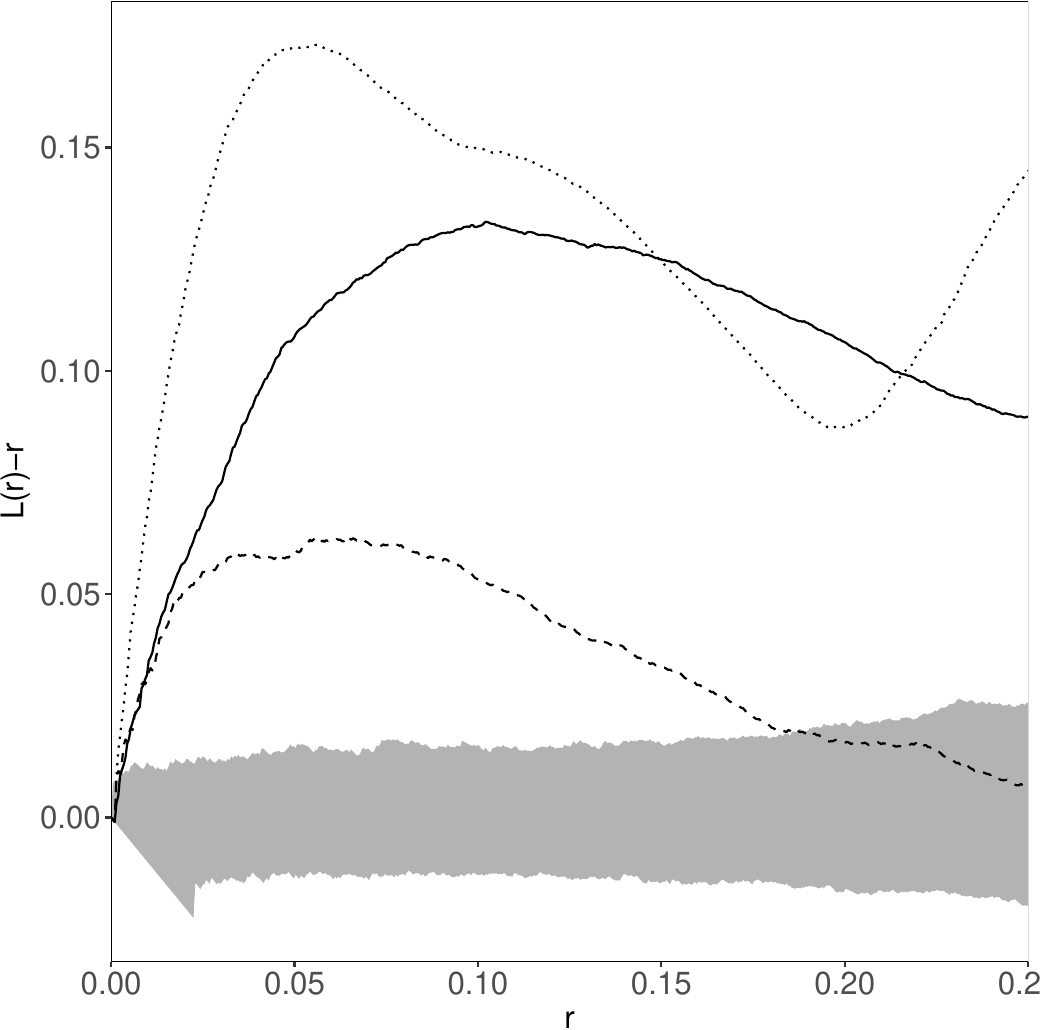}
	\end{minipage}%
	\begin{minipage}{.33\linewidth}
		\includegraphics[width=\linewidth]{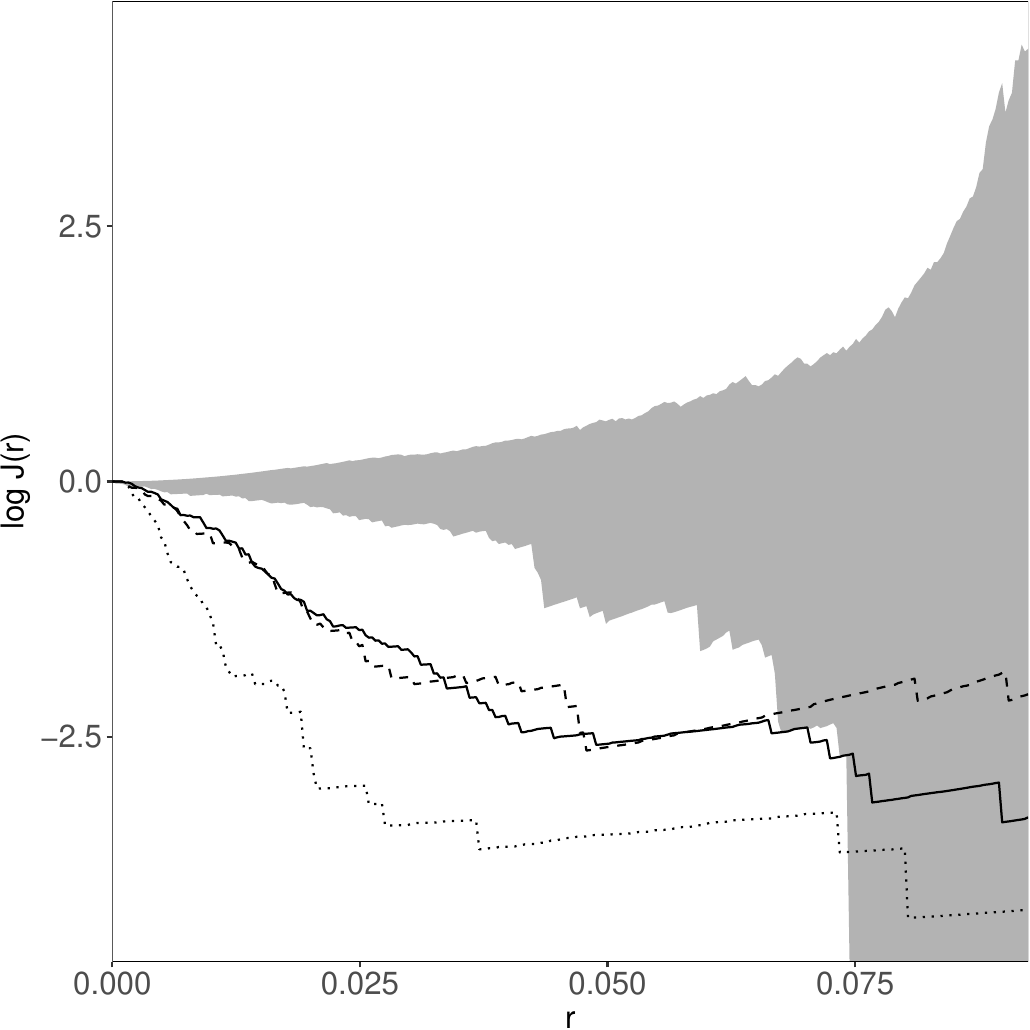}
	\end{minipage}%
	\caption{Empirical PCFs, $L$-functions, and $J$-functions (left to
		right) based on the simulations of $G_0^{\textup{st}}$ from
		Figure~\ref{f:limsim} when the noise processes are either
		determinantal (dashed), Poisson (solid), or weighted permanental
		(dotted). The rows corresponds to Case~1 and 2, respectively. The
		grey regions are 95\% global rank envelopes based on 2499
		simulations of a stationary Poisson process with the same
		intensity as $G_0^{\textup{st}}$.}
	\label{f:limsim_summaries}
\end{figure}

\appendix

\section{Weighted determinantal and permanental \\point processes}\label{a:1}
When defining stationary weighted determinantal/permanental point
processes, the main ingredients are a symmetric function
$C:\mathbb R^d\to\mathbb R$ and a real number $\alpha$. Before
giving the definitions of these point processes we recall the
following.

For a real $n\times n$ matrix $A$ with $(i,j)$-th entry $a_{i,j}$, the
$\alpha$-weighted permanent of $A$ is defined by
\begin{equation*}
	\per_\alpha (A)
	=\sum_\sigma\alpha^{\#\sigma}a_{1,\sigma_1}\cdots a_{n,\sigma_n}
\end{equation*}
where $\sigma$ denotes a permutation of $\{1,\ldots,n\}$ and
$\#\sigma$ is the number of its cycles. This is the usual permanent of
$A$ if $\alpha=1$. Moreover, the $\alpha$-weighted determinant of $A$
is given by
\begin{equation*}
	\det_\alpha (A)=\per_{-\alpha} (-A).
\end{equation*}
This is the usual determinant of $A$ if $\alpha=-1$. Often we just
write $\per_\alpha A$ for $\per_\alpha (A)$, and
$\det_\alpha A$ for $\det_\alpha (A)$.

For any $X_1,\ldots,X_n\in\mathbb R^d$, the $n\times n$ matrix with
$(i,j)$-th entry $C(X_i-X_j)$ is denoted by
$[C](X_1,\ldots,X_n)$. Thus
\begin{equation*}
	\per_\alpha [C](X_1,\ldots,X_n)
	= \sum_\sigma\alpha^{\#\sigma}C(X_1-X_{\sigma_1})\cdots C(X_n-X_{\sigma_n}).
\end{equation*}
Note that the weighted permanent/determinant can be negative if the
mapping $\mathbb R^d\times\mathbb R^d\ni (u,v)\to C(u-v)$ is not
positive semi-definite. When this mapping is positive semi-definite,
$C$ is an auto-covariance function, with corresponding
auto-correlation function $R(x)=C(x)/C(0)$ provided $C(0)>0$.

A locally finite point process $X\subset\mathbb R^d$ has $n$-th order
joint intensity $\rho_X^{(n)}$ for $n=1,2,\ldots$ if for any bounded
and pairwise disjoint Borel sets $A_1,\ldots,A_n\subset\mathbb R^d$,
\begin{equation*}
	\mathrm E\left[N(A_1)\cdots N(A_n)\right]
	=\int_{A_1}\cdots\int_{A_n}\rho_X^{(n)}\left(x_1,\ldots,x_n\right)\,\mathrm dx_1\cdots\,\mathrm dx_n<\infty.
\end{equation*}
Note that $\rho_X^{(n)}$ is unique except for a Lebesgue nullset in
$\mathbb R^{dn}$ (we ignore nullsets in the following).  Thus, if $X$
is stationary, $\rho_X^{(1)}$ is constant and agrees with the
intensity $\rho_X$, and $\rho_X>0$ implies that
$g_X(u-v)=\smash{\rho_X^{(2)}}(u,v)/\rho_X^2$ is the PCF.

If for all $n=1,2,\ldots$, the $n$-th order joint intensity exists and
is
\begin{equation*}
	\rho_X^{(n)}(X_1,\ldots,X_n)=\per_\alpha [C](X_1,\ldots,X_n)
\end{equation*}
we say that $X$ is a stationary $\alpha$-weighted permanental point
process with kernel $C$ and write $X\sim\mathrm{PPP}_\alpha(C)$.
Conditions are needed to ensure the existence of
$\mathrm{PPP}_\alpha(C)$, see \cite{shirai:takahash:03} and
\cite{McCullagh:Moeller:06}.  To exclude the trivial case where $X$ is
empty we assume $\alpha C(0)>0$.  Note that $C$ must be an
auto-covariance function, $\alpha>0$ since $\rho_X=\alpha C(0)$,
and
\begin{align}\label{e:pcfwPPP}
	g_X(x)-1=R(x)^2/\alpha.
\end{align}
This reflects that the process exhibits a positive association between
its points.  In fact, if $C$ is an auto-covariance function and
$k=2\alpha$ is a positive integer, then $X\sim\mathrm{PPP}_\alpha(C)$
exists and it is a Cox process: Conditional on IID zero-mean
stationary Gaussian processes $\Phi_1,\ldots,\Phi_k$ on $\mathbb R^d$
with auto-covariance function $C/2$, we can let $X$ be a Poisson
process with intensity function
$\Lambda(x)=\Phi_1(x)^2+\dots+\Phi_k(x)^2$, $x\in\mathbb R^d$.  In
particular, if $\alpha=1$, then $X$ is the boson process introduced by
\cite{Macchi:75}.

If for all $n=1,2,\ldots$, the $n$-th order joint intensity exists and
is
\begin{equation*}
	\rho_X^{(n)}(G_1,\ldots,G_n)=\det\nolimits_\alpha [C](G_1,\ldots,G_n)
\end{equation*}
we say that $X$ is a stationary $\alpha$-weighted determinantal point
process with kernel $C$ and write $X\sim\mathrm{DPP}_\alpha(C)$.  To
exclude the trivial case where $X$ is empty we assume $\alpha C(0)>0$.
Again $C$ needs to be an auto-covariance function, $\alpha>0$
since $\rho_X=\alpha C(0)$, and
\begin{equation}\label{e:pcfwDPP}
	g_X(x)-1=-R(x)^2/\alpha.
\end{equation}
If $\alpha=1$, then $X$ is the fermion process introduced by
\cite{Macchi:75} (it is usually called the determinantal point
process). We have the following existence result: If $C$ is continuous
and square integrable, existence of $X\sim\mathrm{DPP}_1(C)$ is
equivalent to the Fourier transform of $C$ being bounded by 0 and 1
\cite{LMR15}. When $\alpha$ is a positive integer,
$X\sim\mathrm{DPP}_\alpha(C)$ can be identified with the superposition
$G_1\cup\dots\cup G_\alpha$ of independent processes
$G_i\sim\mathrm{DPP}_\alpha(C/\alpha)$, $i=1,\ldots,\alpha$.  In
general, the process is not well-defined if $0<\alpha<1$, cf.\
\cite{McCullagh:Moeller:06}.

\section{The intensity and PCF of the invariant distribution}\label{a:2}
Let the situation be as in Theorem~\ref{t:main2}. Below we verify
\eqref{e:int_finitecluster} and \eqref{e:PCFequil} holds for $G_n^{\rm st}$ provided $g_G(u-v)$ is a locally integrable function of $(u,v)\in \mathbb{R}^d\times\mathbb{R}^d$.

Note that the $G_n^{\rm st}$ are identically distributed and
$G_0^{\mathrm{st}}=W_0^{\mathrm{st}}\cup Z_0$ where
$W_0^{\mathrm{st}}=\bigcup_{m=1}^\infty\bigcup_{x\in Z_{-m}}W_{-m,x}^{(m)}$,
cf.\ \eqref{e:statext}. Hence, for Borel sets $A\subseteq\mathbb R^d$
with $|A|<\infty$, using similar arguments as in the derivation of \eqref{e:EN}, we obtain
\begin{equation}\label{e:int-statW}
\mathrm E\{\#(W_0^{\mathrm{st}}\cap A)\}=
|A|\rho_Z\frac{\beta p + q}{1 - \beta p - q},
%	\begin{aligned}
%		\mathrm E\{\#(W_0^{\mathrm{st}}\cap A)\}
%		&=\rho_Z\int \biggl\{\sum_{m=1}^\infty \int_A\sum_{k=0}^{m}\binom{m}{k}q^{m-k}(\beta p)^k f^{*k}(y-x)\,\mathrm dy\biggr\}  \,\mathrm dx \\
%		&=|A|\rho_Z\sum_{m=1}^{\infty}(\beta p + q)^m
%		=|A|\rho_Z\frac{\beta p + q}{1 - \beta p - q}
%	\end{aligned}
\end{equation}
%using Fubini's theorem in the second identity, 
so $W_0^{\textup{st}}$ has
intensity
\begin{equation}\label{e:int-W}
\rho_W=\rho_Z\frac{\beta p + q}{1 - \beta p - q}
\end{equation}
whereby it follows that $G_0^{\mathrm{st}}$ has intensity $\rho_G$ as
given by \eqref{e:int_finitecluster}.

Let $A_1,A_2\subseteq\mathbb R^d$ be disjoint Borel sets with
$|A_i|<\infty$, $i=1,2$. Using similar arguments as in the derivation of \eqref{e:EN} (or \eqref{e:int-statW})
 and exploiting the fact that $Z_0,Z_{-1},\ldots$
are IID point processes with a PCF of the form
$g_Z=1+bf_Z*\tilde{f}_Z$ as well as
% well-known properties of the IID Poisson processes
% $Z_0,Z_{-1},\ldots$ and
the independence between $Z_0$ and $W_0^{\mathrm{st}}$, we obtain
\begin{align}
	\MoveEqLeft[3] \mathrm E\{\#(G_0^{\mathrm{st}}\cap
	A_1)\#(G_0^{\mathrm{st}}\cap A_2)\}\nonumber
	\\
	= {} &\rho_Z^2|A_1||A_2|
	+ \rho_Z^2\int_{A_1}\int_{A_2}bf_Z*\tilde{f}_Z(x_1-x_2)\, \dd x_1\,
	\dd x_2
	+2\rho_Z\rho_W|A_1||A_2|\label{e:ac:dddd1}
	\\
	&+\sum_{m_1=1}^\infty\sum_{m_2=1:\,m_1 \neq m_2}^\infty
	\rho_Z^2(\beta p + q)^{m_1+m_2}|A_1||A_2| \label{e:ac:dddd2}
	\\
	&
	\begin{aligned}
	&+\sum_{m=1}^\infty\rho_Z^2(\beta p + q)^{2m}|A_1||A_2|
	\\
	&+\sum_{m=1}^\infty\rho_Z^2\int_{A_1}\int_{A_2}bf_Z*\tilde{f}_Z\\
	&\hspace{5mm}*\sum_{k_1=0}^{m}\sum_{k_2=0}^{m}\binom{m}{k_1}\binom{m}{k_2}q^{2m-k_1-k_2}(\beta p)^{k_1+k_2}f^{*k_1}*\tilde{f}^{*k_2}(y_1-y_2)\,
	\dd y_1\,\dd y_2
	\end{aligned}
	\label{e:ac:dddd3}
	\\
	&+\sum_{m=1}^\infty\mathrm E\biggl\{\sum_{x\in Z_{-m}}\#(W_{-m,x}^{(m)}\cap
	A_1)\#(W_{-m,x}^{(m)}\cap A_2)\biggr\}. \label{e:ac:dddd4}
\end{align}
Here, 
\begin{itemize}
	\item the first two terms of \eqref{e:ac:dddd1} corresponds to pairs of
	points from $Z_0$ with one point falling in $A_1$ and the other in
	$A_2$;
	\item the third term corresponds to pairs of points either from
	$Z_0\cap A_1$ and $W_0^{\mathrm{st}}\cap A_2$ or from $Z_0\cap A_2$
	and $W_0^{\mathrm{st}}\cap A_1$;
	\item the term in \eqref{e:ac:dddd2} corresponds to pairs of points, with one point
	falling in $A_1$ and the other in $A_2$ of two families initiated by
	ancestors from different generations;
	\item the two terms in \eqref{e:ac:dddd3} corresponds to pairs of points, with one point
	falling in $A_1$ and the other in $A_2$
	from two different families initiated by ancestors from the same generation;
	\item the term in \eqref{e:ac:dddd4} corresponds to pairs of
	points from the same family, falling in $A_1$ and $A_2$, respectively.
\end{itemize}
%the first two term of \eqref{e:ac:dddd1} corresponds to pairs of
%points from $Z_0$ with one point falling in $A_1$ and the other in
%$A_2$, the second term corresponds to pairs of points either from
%$Z_0\cap A_1$ and $W_0^{\mathrm{st}}\cap A_2$ or from $Z_0\cap A_2$
%and $W_0^{\mathrm{st}}\cap A_1$. Moreover, the term in
%\eqref{e:ac:dddd2} corresponds to pairs of points, with one point
%falling in $A_1$ and the other in $A_2$ of two families initiated by
%ancestors from different generations, while the term in
%\eqref{e:ac:dddd3} corresponds to such pairs of points in two
%different families initiated by ancestors from the same generation,
%and finally the term in \eqref{e:ac:dddd4} corresponds to pairs of
%points from the same family, falling in $A_1$ and $A_2$, respectively.
% Here, in the right side of \eqref{e:dddd1}, the first term
% corresponds to pairs of points from $Z_0$ with one point falling in
% $A_1$ and the other in $A_2$, the second term corresponds to pairs
% of points either from $Z_0\cap A_1$ and $W_0^{\mathrm{st}}\cap A_2$
% or from $Z_0\cap A_2$ and $W_0^{\mathrm{st}}\cap A_1$, and the third
% term corresponds to pairs of points from $W_0^{\mathrm{st}}$ coming
% from different families and with one point falling in $A_1$ and the
% other in $A_2$.
Using \eqref{e:int_finitecluster} and \eqref{e:int-W}, we observe that
\eqref{e:ac:dddd1}--\eqref{e:ac:dddd3} simplify to
\begin{equation}\label{e:some_clever_name}
\begin{aligned}
\rho_G^2|A_1||A_2| &+ \sum_{m=0}^\infty\rho_Z^2
\int_{A_1}\int_{A_2}bf_Z*\tilde{f}_Z \\
&\hspace{-3mm}*\sum_{k_1=0}^{m}\sum_{k_2=0}^{m}\binom{m}{k_1}\binom{m}{k_2}q^{2m-k_1-k_2}(\beta p)^{k_1+k_2}f^{*k_1}*\tilde{f}^{*k_2}(y_1-y_2)\,
\dd y_1\,\dd y_2
\end{aligned}
\end{equation}
and the term in \eqref{e:ac:dddd4} is equal to
\begin{equation}\label{e:dddd3}
	\begin{aligned}
		\rho_Z\sum_{m=1}^{\infty}\iiiint\!\!\!\int_{A_1}\int_{A_2}\big(
		&(\beta pf+q\delta_0)^{*i}(y-x) \\
		&\cdot[c(\beta p)^2f(\tilde{y}_1-y)f(\tilde{y}_2-y) \\
		&\hspace{3mm}+\beta pq\left\{f(\tilde{y}_1-y)\delta_0(\tilde{y_2}-y) + \delta_0(\tilde{y}_1-y)f(\tilde{y_2}-y) \right\}] \\
		&\cdot(\beta pf+q\delta_0)^{*(m-1-i)}(y_1-\tilde{y}_1) \\
		&\cdot(\beta pf+q\delta_0)^{*(m-1-i)}(y_2-\tilde{y}_2)
		\big)
		\,\dd y_1\,\dd y_2\,\dd \tilde{y}_1\,\dd \tilde{y}_2\,\dd y\,\dd x.
	\end{aligned}
\end{equation}
In \eqref{e:dddd3}, $y$ corresponds to an $i$-th generation point in the family
	initiated by $x\in Z_{-m}$, $c(\beta p)^2 + 2\beta pq$ is the expected
	number of pairs of points $\tilde{y}_1$ and $\tilde{y}_2$ which are children of $y$, and $y_1$ and $y_2$ are the $(m-1-i)$-th generation offspring of $\tilde{y}_1$ and $\tilde{y}_2$, respectively. Using Fubini's theorem together with \eqref{e:int_finitecluster}, after straight forward calculations, \eqref{e:dddd3} reduces to
\begin{align*}
	&\rho_G\int_{A_1}\int_{A_2}\sum_{i=0}^{\infty}\left\{c(\beta p)^2f*\tilde{f} + \beta pq(f + \tilde{f})\right\}\\
	&\hspace{15mm}*\sum_{k_1=0}^{i}\sum_{k_2=0}^{i}\binom{i}{k_1}\binom{i}{k_2}q^{2i-k_1-k_2}(\beta p)^{k_1+k_2}f^{*k_1}*\tilde{f}^{*k_2}(y_1-y_2)\,\dd y_1\,\dd y_2
\end{align*}
%where \eqref{e:int_finitecluster} has been used. 
Combining this result with \eqref{e:some_clever_name} we finally see that $G_0^{\mathrm{st}}$ has PCF $g_G$ as given
by \eqref{e:PCFequil}.

\section{Simulating the limiting process}\label{a:3}
This appendix presents an approximate simulation procedure for
simulating a special case of $G_0^{\textup{st}}$ on a bounded region
$R\subset\mathbb{R}^d$. It is available in \verb|R| through the
package \verb|icpp|, which can be obtained at
\url{https://github.com/adchSTATS/icpp}. The implementation utilizes
existing functions from the packages \verb|spatstat| and
\verb|RandomFields| to simulate the noise process.

For simplicity and specificity we make the following assumptions. Let the situation be as in
Theorem~\ref{t:main2}, but with $q = 0$ and let $f\sim N_d(\sigma^2)$ with
$\sigma>0$. Also, without loss of generality, assume no thinning
(i.e., $p = 1$).  Let
$R_{\oplus r}=\{\xi\in\mathbb R^d:b(\xi,r)\cap R\not=\emptyset\}$
where $b(\xi, r)$ is a closed ball with centre $\xi$ and radius
$r\ge0$. Denote $n$ the number of iterations in our approximate
simulation algorithm, that is, $-n$ is the starting time when ignoring
what happens previously. Note that $\sqrt{n}\sigma$ is the standard
deviation of the $n$th convolution power of $f$. To account for edge
effects, let $r=4\sqrt{n}\sigma$ where 4 is an arbitrary non-negative
value ensuring that a point of $G_{-n}^{\textup{st}}\setminus R_{\oplus r}$
would generate a $n$th generation offspring in $R$ with very low
probability, at most $1/15787$. In the approximate simulation
procedure, we ignore those points of $G_0^{\textup{st}}\cap R$ which are
generated by an $i$th generation ancestor $x$ when $i<-n$ or both
$-n\le i<0$ and $x\not\in R_{\oplus 4\sqrt{-i}\sigma}$.
% This is a simplified version of our algorithm in pseudocode:
This is our algorithm in
pseudocode %for approximate simulation of $G_0^{\rm st}$
where ``parallel-for'' means a parallel for loop:

\vspace{0.2cm}

%%%%%% ALGORITHM 1 %%%%%
% \noindent\hrulefill
%
%\noindent\textbf{Algorithm for simulating $G_0^{\rm st}$} \\
%\textbf{for} $i = -n$ to 0 \textbf{do} \\
% \indent Simulate: $Z'_{i} = Z_{i} \cap R_{\oplus 4\sqrt{-i}\sigma}$. \\
%	\indent \textbf{if} $i\ne 0$ \textbf{then} \\
% \indent\indent For each $x\in Z'_{i}$ generate its $(-i)$th generation offspring, ${W}_{i, x}^{(-i)}$. \\
% \indent\indent Let ${W'}_{i}^{(-i)}$ denote the union of ${W}_{i, x}^{(-i)}$ over $x\in Z'_{i}$ restricted to $R$. \\
%	\indent	\textbf{end if} \\
% \textbf{end for}\\
% \textbf{return} the union of
% $Z'_0, {W'}_{-1}^{(1)}, \ldots, {W'}_{-n}^{(n)}$.
%
% \noindent\hrulefill
%%%%%%%%%%%%%%%%%%%%%%%

% We also present an alternate algorithm that takes advantage of the
% independence between the noise processes as well as the iterative
% structure.

%%%%% ALGORITHM 2 %%%%%
\noindent\hrulefill

% \noindent\textbf{Algorithm for simulating $G_0^{\rm st}$} \\
\noindent\textbf{parallel-for} $i = -n$ to 0 \textbf{do} \\
\indent simulate $Z'_{i} := Z_{i} \cap R_{\oplus 4\sqrt{-i}\sigma}$\\
\textbf{end parallel-for} \\
set $O := Z'_{-n}$ \\
\textbf{if} $n \ne 0$ \textbf{then} \\
\indent\textbf{for} $i = -(n-1)$ to 0 \textbf{do} \\
\indent\indent \textbf{parallel-for} $x\in O$ \textbf{do}\\
\indent\indent\indent simulate the $1$st generation offspring process, $O_x$, with parent $x$ \\
\indent\indent\textbf{end parallel-for} \\
\indent\indent set $O := Z'_{i}\bigcup \left(\bigcup_{x\in O} O_x\cap R_{\oplus 4\sqrt{-i}\sigma}\right)$ \\
\indent\textbf{end for} \\
\textbf{end if} \\
\textbf{return} $O$

\noindent\hrulefill
%%%%%%%%%%%%%%%%%%%%%%%

\vspace{0.2cm}

% The later package is used to determine the intensity of the weighted
% permanental point process (which is regarded as a Cox point
% process).  For efficiency, we actually simulate the current
% generation noise process; generate only the first generation
% offspring; and get the union of this point pattern with the noise
% process of the following generation and so on. Also, since the noise
% processes are independent, we do simulation of these in parallel.
% The proposed simulations procedure is available through the \verb|R|
% package \verb|icpp|, which may be obtained at
% \url{https://github.com/adchSTATS/icpp}. The implementation takes
% advantage of existing functions from the packages \verb|spatstat|
% and \verb|RandomFields| to simulate the noise process. The later
% package is used to determine the intensity of the weighted
% permanental point process (which is regarded as a Cox point
% process).  Our implementation is done in \verb|R| and for simulating
% the noise processes we use existing functions from the packages
% \verb|spatstat| and \verb|RandomFields|.  The later package is used
% to determine the intensity of the weighted permanental point process
% (which is regarded as a Cox point process). Simulating determinantal
% point patterns with many points is often time consuming, which in
% turn makes our algorithm time consuming.  However, as the noise
% processes are independent, we may do these simulations in parallel
% for higher efficiency.

Note that $\rho_Z\sum_{i = 0}^{n}(\beta p)^i$ is the intensity of the
stationary point process obtained by ignoring those points of
$G_0^{\textup{st}}$ which are generated by an $i$th generation ancestor
with $i<-n$. We base the choice of $n$ on this fact by considering a
precision parameter $\varepsilon>0$ and letting
\begin{align*}
	n = \sup\biggl\{m\in\{1,2,\ldots\}:\norm[\Big]{\rho_Z\sum_{i = 0}^{m}(\beta
		p)^i - \rho_{G}}\le \varepsilon\biggr\}.
\end{align*}
To exemplify, let $\rho_G = 100$ and $\beta p = 0.8$ implying that
$\rho_Z = 20$, and let $\varepsilon = 2.22\times 10^{-16}$, then
$n = 159$. If instead $\beta p = 0.99$, then $n = 3609$.

\section*{Acknowledgements}
Supported by The Danish Council for Independent Research | Natural
Sciences, grant DFF -- 7014-00074 ``Statistics for point processes in
space and beyond'', and by the ``Centre for Stochastic Geometry and
Advanced Bioimaging'', funded by grant 8721 from the Villum
Foundation.  We thank Ina Trolle Andersen, Yongtao Guan, Ute Hahn, Henrike H{\"a}bel, Eva
B.~Vedel Jensen, and Morten Nielsen for helpful comments.

\bibliographystyle{plain}
\bibliography{references}
\end{document}